\documentclass[reqno, 11pt]{amsart}
\usepackage{verbatim}

\newcommand{\mysection}[1]{\section{#1}
\setcounter{equation}{0}}

\newtheorem{theorem}{Theorem}[section]
\newtheorem{corollary}[theorem]{Corollary}
\newtheorem{lemma}[theorem]{Lemma}

\theoremstyle{definition}
\newtheorem{remark}[theorem]{Remark}

\theoremstyle{definition}

\theoremstyle{definition}

\makeatletter
\def\dashint{\operatorname%
{\,\,\text{\bf--}\kern-.98em\DOTSI\intop\ilimits@\!\!}}
\makeatother

\renewcommand{\epsilon}{\varepsilon}

\def\bR{\mathbb{R}}

\def\bH{\mathbb{H}}

\def\bP{\mathbb{P}}

\def\hT{\hat{T}}

\def\ff{\mathfrak{f}}
\def\fg{\mathfrak{g}}

\def\cA{\mathcal{A}}
\def\cB{\mathcal{B}}

\def\cD{\mathcal{D}}

\def\cF{\mathcal{F}}
\def\cG{\mathcal{G}}
\def\cH{\mathcal{H}}
\def\cP{\mathcal{P}}

\def\cL{\mathcal{L}}

\newcommand{\ip}[1]{\left\langle#1\right\rangle}
\newcommand{\set}[1]{\left\{#1\right\}}

\newcommand{\Div}{\operatorname{div}}

\begin{document}
\title[Parabolic and elliptic systems]{Gradient estimates for parabolic and elliptic systems from linear laminates}

\author[H. Dong]{Hongjie Dong}
\address[H. Dong]{Division of Applied Mathematics, Brown University,
182 George Street, Providence, RI 02912, USA}
\email{Hongjie\_Dong@brown.edu}
\thanks{H. Dong was partially supported by the NSF under agreements DMS-0800129 and DMS-1056737.}

\subjclass[2010]{35R05,35J55}

\keywords{Second-order systems, partially H\"older coefficients, partially Dini coefficients}

\begin{abstract}
We establish several gradient estimates for second-order divergence type parabolic and elliptic systems. The coefficients and data are assumed to be H\"older or Dini continuous in the time variable and all but one spatial variables. This type of systems arises from the problems of linearly elastic laminates and composite materials. For the proof, we use Campanato's approach in a novel way. Non-divergence type equations under a similar condition are also discussed.
\end{abstract}

\maketitle

\setcounter{tocdepth}{1}
\tableofcontents

\mysection{Introduction}
                                        \label{secIntro}

We consider second-order divergence type parabolic and elliptic systems with coefficients and data which are irregular in one of spatial directions. This type of systems arises from the problems of linearly elastic laminates and composite materials. We are interested in the local regularity of the gradient of weak solutions to these systems.

Problems of this kind have been studied by many authors; see, for instance, \cite{CKC, BASL, BV, LiVo, LiNi}. 
In \cite{CKC}, Chipot, Kinderlehrer, and Vergara-Caffarelli considered the weak variational formulation of the equilibrium problem of a linear laminates, i.e. a domain in $\bR^d$ consisting of a finite number $M$ of linearly elastic, homogeneous, parallel laminae. They proved that any weak solution $u$ of the uniformly elliptic system
$\Div(A\nabla u)=f$
is actually locally Lipschitz, under the conditions that $f$ is in $H^k(\Omega), k\ge [d/2]$, and the coefficients matrix $A$ are constants in each parallel laminae. They also showed that the $W^{1,\infty}$ norm of $u$ is independent of the number $M$, so that in the limiting case $A$ are allowed to be functions of one direction alone. In \cite{LiVo}, Li and Vogelius considered scalar elliptic equations for a single real function $u$:
$$
\Div(A\nabla u)=f+\Div g
$$
in a domain which consists a finite number $M$ of disjoint sub-domains with $C^{1,\alpha}$ boundaries. This equation models deformations in composite media such as fiber-reinforced materials. The matrix $A$ and data are assumed to be H\"older continuous up to the boundary in each sub-domains, but may have jump discontinuities across the boundaries of the sub-domains. Under these assumptions, the authors derived global $W^{1,\infty}$ and piecewise $C^{1,\delta}$ estimates of the solution $u$ for $\delta\in (0,\frac{\alpha}{d(\alpha+1)}]$. Their results were later extended to elliptic systems for vector-valued function $u$ by Li and Nirenberg \cite{LiNi} under the same conditions, and the range of $\delta$ was also relaxed to $(0,\frac{\alpha}{2(\alpha+1)}]$. The bounds obtained in \cite{LiVo, LiNi}, however, may depend on the number $M$. The corresponding problem for parabolic systems is more complicated. In a forthcoming paper, Li and Li \cite{LiLi} further extend some results in \cite{LiNi} to parabolic systems under an additional assumption that the coefficients and data are at least twice differentiable in $t$. See also \cite{FKNN} for a related result on parabolic systems.

We would like to mention two recent papers \cite{DongSeickK09} and \cite{TianWang} on ``partial Schauder'' estimates. In \cite{DongSeickK09}, Dong and Kim considered both divergence and non-divergence form second-order scalar elliptic and parabolic equations. They proved that if the coefficients and data are H\"older continuous in some directions, then derivatives of solutions in these directions are H\"older continuous in the same directions. By using a different method, Tian and Wang \cite{TianWang} proved similar results for non-divergence form elliptic equations with coefficients and data Dini continuous in some variables. Under certain conditions, their results also extend to second-order fully nonlinear equations. An interesting question is how much regularity one can expect in the ``bad'' directions. In this paper we address this question when the ``bad'' direction is one-dimensional. We note that the proofs in \cite{DongSeickK09} and \cite{TianWang} do not apply to systems since the maximum principle is used in both papers. In the case that the coefficients are regular in all directions, a similar problem was studied long time ago by Fife \cite{Fife}.

In this paper, we are concerned with parabolic and elliptic systems:
\begin{align}
                            \label{parabolic0}
\cP u&:=-u_t+D_\alpha(A^{\alpha\beta}D_\beta u)+D_\alpha(B^{\alpha}u)+\hat B^{\alpha}D_\alpha u+ C u=\Div g + f,\\
\cL u&:=D_\alpha(A^{\alpha\beta}D_\beta u)+D_\alpha(B^{\alpha}u)+\hat B^{\alpha}D_\alpha u+ C u=\Div g + f.\nonumber
\end{align}
The coefficients of $\cP$ and $\cL$ are assumed to be bounded, and the operators are uniformly nondegenerate.
The aim of our paper is to obtain optimal regularity of weak solutions when coefficients are assumed to be regular in the time variable and all but one spatial variables. To be more precise, let us denote a typical point in ${\bR}^{d+1}$ by $z=(t,x)$, where $x=(x^1,\cdots,x^d):=(x',x^d)$ and $z':=(t,x')$. We shall prove that if the coefficients and data are Dini continuous in $z'$, then any weak solution $u$ to \eqref{parabolic0} is locally Lipschitz in all spatial variables, $C^{1/2}$ in $t$, and $D_{x'}u$ and $\hat U:=A^{d\beta}D_\beta u+B^d u-g_d$ are continuous; see Theorem \ref{thm1} below for more precise statement. We also prove that a H\"older regularity assumption in $z'$ on the coefficients and data gives a better regularity of $u$. In particular, $D_{x'} u$ and $\hat U$ are H\"older in all variables; see Theorem \ref{thm2}.

In the special case that the domain consists of $M$ parallel laminate sub-domains as in \cite{CKC} mentioned above, we show that if the coefficients and data are regular in each sub-domain and may have jump discontinuous across the boundaries, then $Du$ is regular in each sub-domain up to the boundary; see Remark \ref{rem5.1} i). Thus our results generalize the aforementioned results in \cite{CKC} by allowing more general coefficients and also deriving optimal $C^{1,\delta}$ estimate for $\delta\in (0,1)$. Unlike \cite{LiVo, LiNi}, we do not impose any restriction on $\delta$, and the bounds of various norms are independent of $M$. However, it should be pointed out that although we allow subdomains to have curved boundaries (see Remark \ref{rem5.1} ii)), the geometry of the domain considered in \cite{LiVo, LiNi} is more general than in the current paper.

Our arguments are quite different from those in \cite{LiVo, LiNi, LiLi} and \cite{DongSeickK09, TianWang}. Let us give a brief description as follows. The proofs below are based on Campanato's approach, which was used previously, for instance, in \cite{Giaq83,Lieberman92}. The main step of Campanato's approach is to show the mean oscillations of $Du$ in balls vanish in certain order as the radii of balls go to zero. However, we are not able to follow this approach in the usual way due to the lack of regularity of $u$ in the $x^d$-direction. To overcome this difficulty, we appeal to a recent result in \cite{DK09} regarding the $L_p$ estimate for systems with partially VMO coefficients. A crucial step in the proof is to deduce from this result some interior H\"older regularity of $D_{x'}u$ and $U:=A^{d\beta}D_\beta u$ for parabolic systems with coefficients depending on $x^d$ alone. We then use some perturbation arguments on $D_{x'}u$ and $U$ together with a certain decomposition of $u$ to get the desired estimates.

By using a similar idea, we also obtain the corresponding results for scalar non-divergence form parabolic equations:
$$
P u:=-u_t+a^{\alpha\beta}D_{\alpha\beta}u+b^\alpha D_\alpha u+c u=f.
$$
We prove that if the coefficients and $f$ are Dini continuous in $z'$, then any solution $u$ of the above equation is $C^1$ in $t$, $C^{1,1}$ in $x$, and $u_t$ and $D_{xx'} u$ are continuous; see Theorem \ref{thm3}. Under the stronger condition that the coefficients and $f$ are H\"older continuous in $z'$, we obtain additionally that $u_t$ and $D_{xx'} u$ are H\"older continuous in all variables; see Theorem \ref{thm4}. These theorems generalize some results in \cite{Fife, DongSeickK09, TianWang} for the Poisson equation. In the case that the domain consists of $M$ parallel laminate sub-domains and the coefficients and $f$ are regular in each sub-domain, we show that the second derivative of $u$ in the $x^d$-direction is also H\"older continuous up to the boundary in each sub-domain.

As mentioned above, some of our estimates rely on recent work about $L_p$-regularity for elliptic and parabolic equations (systems) with leading coefficients VMO in some of the independent variables; see Section \ref{sec3}. This series of work was initiated by Krylov in \cite{Krylov_2005}. For further developments in this direction, we also refer the reader to \cite{KK2, DK09, Dong09a} and references therein.


The organization of this paper is as follows. In Section \ref{sec2}, we state our main
theorems for divergence form systems and introduce some notation. We prove some auxiliary estimates in Section \ref{sec3}. The proofs of main theorems
are given in Section \ref{sec4} and Section \ref{sec5}. Finally we treat non-divergence form scalar equations in Section \ref{sec6}.

\mysection{Notation and main results}
                                            \label{sec2}

We are concerned with parabolic systems
\begin{equation}
                                                \label{parabolic}
\cP u:=-u_t+D_\alpha(A^{\alpha\beta}D_\beta u)+D_\alpha(B^{\alpha}u)+\hat B^{\alpha}D_\alpha u+ C u=\Div g + f,
\end{equation}
where $g=(g_1,g_2,\cdots,g_d)$.
The coefficients $A^{\alpha\beta}$, $B^\alpha$, $\hat B^\alpha$, $C$ are $n \times n$ matrices,
which are bounded by a positive constant $K$, and
the leading coefficient matrices $A^{\alpha\beta}$ are uniformly elliptic with ellipticity constant $\nu$:
$$
\nu|\xi|^2\le A^{\alpha\beta}_{ij}\xi^\alpha_i\xi^\beta_j,\quad |A^{\alpha\beta}|\le \nu^{-1}
$$
for any $\xi=(\xi^\alpha_i)\in \bR^{d\times n}$.
Here
$$
u = (u^1, \cdots, u^n)^{\text{tr}},
\quad
g_{\alpha} = (g^1_\alpha, \cdots, g^n_\alpha)^{\text{tr}},
\quad
f = (f^1,\cdots,f^n)^{\text{tr}}
$$
are (column) vector-valued functions. Throughout the paper, the summation convention over repeated indices is used. We also consider the following elliptic system
\begin{equation}
                                                \label{elliptic}
\cL u:=D_\alpha(A^{\alpha\beta}D_\beta u)+D_\alpha(B^{\alpha}u)+\hat B^{\alpha}D_\alpha u+ C u=\Div g + f.
\end{equation}
In this case $A^{\alpha\beta}$, $B^\alpha$, $\hat B^\alpha$, $C$, $g$, and $f$
are independent of $t$ and satisfy the same conditions as in the parabolic case.
\subsection{Notation}
By $Du=(D_i u)$ and $D^{2}u=(D_{ij}u)$ we mean the gradient and the Hessian matrix
of $u$. On many occasions we need to take these objects relative to only
part of variables. We also use the following notation:
$$
D_tu=u_t,\quad D_{x'}u=u_{x'},\quad
D_{xx'}u=u_{xx'}.
$$

Set
$$
B_r'(x') = \{ y \in \bR^{d-1}: |x'-y'| < r\}, \quad
B_r(x) = \{ y \in \bR^d: |x-y| < r\},
$$
$$
Q_r'(t,x) = (t-r^2,t) \times B_r'(x'),\quad
Q_r(t,x) = (t-r^2,t) \times B_r(x),
$$
and
$$
B_r'=B_r'(0),\quad
B_r = B_r(0),\quad
Q_r'=Q_r'(0,0),\quad
Q_r=Q_r(0,0).
$$

By $N(d,p,\cdots)$ we mean that $N$ is a constant depending only
on the prescribed quantities $d, p,\cdots$.
For a (matrix-valued) function $f(t,x)$ in $\bR^{d+1}$, we set
\begin{equation*}
(f)_{\cD} = \frac{1}{|\cD|} \int_{\cD} f(t,x) \, dx \, dt
= \dashint_{\cD} f(t,x) \, dx \, dt,
\end{equation*}
where $\cD$ is an open subset in $\bR^{d+1}$ and $|\cD|$ is the
$d+1$-dimensional Lebesgue measure of $\cD$.

\subsection{Lebesgue spaces}
For $p\in (1,\infty)$, we denote
\begin{align*}
W_{p}^{1,2}(\cD)&=
\set{u:\,u,u_t,Du,D^2u\in L_{p}(\cD)}.
\end{align*}
We also denote $\bH^{-1}_{p}(\cD)$ to be the space consisting of all functions $u$ satisfying
$$
\inf\set{\|g\|_{L_{p}(\cD)}+\|h\|_{L_{p}(\cD)}\,|\,u=\Div g+h}<\infty.
$$
It is easy to see that $\bH^{-1}_{p}(\cD)$ is a Banach space. Naturally, for any $u\in \bH^{-1}_{p}(\cD)$, we define the norm
\begin{equation*}
\|u\|_{\bH^{-1}_{p}(\cD)}=\inf\set{\|g\|_{L_{p}(\cD)}+\|h\|_{L_{p}(\cD)}\,|\,u=\Div g+h}.
\end{equation*}
We also define
$$
\cH^{1}_{p}(\cD)=
\set{u:\,u,Du \in L_{p}(\cD),u_t\in \bH^{-1}_{p}(\cD)}.
$$
We use the abbreviations $W_{p}^{1,2}=W_{p}^{1,2}(\bR^{d+1})$ and $\cH^{1}_{p}=\cH^{1}_{p}(\bR^{d+1})$, etc.

\subsection{Partially VMO and partially Dini spaces} For a function $u$ in $\cD$, we define  its modulus of continuity $\omega_{u,z'}$ (in the mean) with respect to $z'$ by
\begin{multline*}
\omega_{u,z'}(R)\\=\sup_{\substack{z_0\in \bR^{d+1}\\r\le R}}\left(
|Q_r|^{-2}\int_{Q_r(z_0)\cap \cD}\int_{Q_r(z_0)\cap \cD}|u(t,x',x^d)-u(s,y',x^d)|^2\,dy\,ds\,dx\,dt
\right)^{\frac 1 2}.
\end{multline*}
We say $u$ is partially VMO with respect to $z'$ if $\omega_{u,z'}(R)\to 0$ as $R\to 0$.

We call a continuous increasing function $\omega$  on $\overline{\bR^+}$ a Dini function if $\omega(0)=0$ and for any $t>0$
$$
I[\omega](t):=\int_0^t \omega(s)/s\,ds<\infty.
$$
We say function $u$ in $\cD$ is partially Dini continuous with respect to $z'$ if its modulus of continuity $\omega_{u,z'}$ in $z'$
is a Dini function. In this case, we write $u\in C^{\text{Dini}}_{z'}(\cD)$. Clearly, any function in $C^{\text{Dini}}_{z'}(\cD)$ is partially VMO with respect to $z'$.

In a similar way, in the time independent case we define $\omega_{u,x'}$ and the space $C^{\text{Dini}}_{x'}$.

We note that our definition of Dini continuity is slightly different from the usual definition, where the modulus of continuity is measured in the uniform sense.

\subsection{H\"older spaces}
For $\delta\in (0,1]$, we denote the $C^{\delta/2,\delta}$ semi-norm by
$$
[u]_{\delta/2,\delta;\cD}:=\sup_{\substack{(t,x),(s,y)\in \cD\\ (t,x)\neq (s,y)}}\frac {|u(t,x)-u(s,y)|}{|t-s|^{\delta/2}+|x-y|^\delta},
$$
and the $C^{\delta/2,\delta}$ norm by
$$
|u|_{\delta/2,\delta;\cD}:=[u]_{\delta/2,\delta;\cD}+
|u|_{0;\cD},
$$
where $|u|_{0;\cD}=\sup_{\cD}|u|$. Next we define
\begin{align*}
[u]_{(1+\delta)/2,1+\delta;\cD}&:=
[Du]_{\delta/2,\delta;\cD}+\ip{u}_{1+\delta;\cD},\\
|u|_{(1+\delta)/2,1+\delta;\cD}&:=[u]_{(1+\delta)/2,1+\delta;\cD}+|u|_{0;\cD}+|Du|_{0;\cD},
\end{align*}
where
\[
\ip{u}_{1+\delta;\cD}:=\sup_{\substack{(t,x),(s,x)\in\cD\\ t\neq s}}\frac {|u(t,x)-u(s,x)|}{|t-s|^{(1+\delta)/2}}.
\]
By $C^{(1+\delta)/2,1+\delta}(\cD)$ we denote the set of all bounded measurable functions $u$ for which the derivatives $D u$  are continuous and bounded in $\cD$ and $[u]_{(1+\delta)/2,1+\delta;\cD}<\infty$.

We define a partial H\"older semi-norm with respect to $z'$ as
\[
[u]_{z',\delta/2,\delta;\cD}:=\sup_{\substack{(t,x),(s,y)\in \cD\\x^d=y^d,(t,x)\neq (s,y)}}\frac {|u(t,x)-u(s,y)|}{|t-s|^{\delta/2}+|x-y|^\delta},
\]
and the corresponding norm as
\[
|u|_{z',\delta/2,\delta;\cD}:=[u]_{z',\delta/2,\delta;\cD}+|u|_{0;\cD}.
\]
By $C^{\delta/2,\delta}_{z'}(\cD)$ we denote the set of all bounded measurable functions $u$ on $\cD$ for which $[u]_{z', \delta/2,\delta;\cD}<\infty$. Similarly, we define \[
[u]_{z',(1+\delta)/2,1+\delta;\cD}:=
[D_{x'}u]_{z',\delta/2,\delta;\cD}+\ip{u}_{1+\delta;\cD},
\]
and the space $C_{z'}^{(1+\delta)/2,1+\delta}(\cD)$. In the time-independent case, we also define $[\cdot]_{x',\delta}$, $|\cdot|_{x',\delta}$ and the space $C^\delta_{x'}$ in a similar fashion.

\subsection{Main results}

We state the main results of the paper concerning divergence form parabolic systems. Roughly speaking, the first theorem reads if the coefficients and data are Dini continuous in $z'$, then any weak solution $u$ to \eqref{parabolic} is Lipschitz in all spatial variables and $1/2$-H\"older in $t$.
\begin{theorem}
                            \label{thm1}
Let $A\in C_{z'}^{\text{Dini}}$, $B\in C_{z'}^{\text{Dini}}$, $f\in L_\infty(Q_1)$ and $g\in C_{z'}^{\text{Dini}}(Q_1)$. Assume that $u$ is a weak solution to \eqref{parabolic} in $Q_1$. Then we have $u\in C^{1/2,1}(Q_{1/2})$ and
\begin{equation}
                                    \label{eq4.30}
|u|_{1/2,1;Q_{1/2}}\le N(I[\omega_{g,z'}](1)+|g|_{0;Q_1}+|f|_{0;Q_1}
+\|u\|_{L_2(Q_1)}),
\end{equation}
where $N=N(d,n,\nu,K,\omega_{A,z'},\omega_{B,z'})$. Moreover, $D_{x'}u$ and $\hat U$ are continuous in $\overline{Q_{1/2}}$,
where
$$
\hat U:=A^{d\beta}D_\beta u+B^d u-g_d.
$$
\end{theorem}

In the next theorem, we show that a partial H\"older regularity assumption on the coefficients and data gives a better regularity of $u$. In particular, the spatial derivatives of $u$ in $x'$ and $\hat U$ are H\"older continuous in all variables.

\begin{theorem}
                            \label{thm2}
Let $\delta\in (0,1)$, $A\in C_{z'}^{\delta/2,\delta}$, $B\in C_{z'}^{\delta/2,\delta}$, $f\in L_\infty(Q_1)$ and $g\in C_{z'}^{\delta/2,\delta}(Q_1)$. Assume that $u$ is a weak solution to \eqref{parabolic} in $Q_1$. Then we have $u\in C^{1/2,1}(Q_{1/2})$, $D_{x'}u,\hat U\in C^{\delta/2,\delta}(Q_{1/2})$ and
\begin{multline}
                                    \label{eq4.20}
|Du|_{0;Q_{1/2}}+\ip{u}_{1+\delta;Q_{1/2}}+
[D_{x'}u]_{\delta/2,\delta;Q_{1/2}}+[\hat U]_{\delta/2,\delta;Q_{1/2}}\\
\le N(|g|_{z',\delta/2,\delta;Q_{1}}+|f|_{0;Q_{1}}
+\|u\|_{L_2(Q_1)}),
\end{multline}
where $N=N(d,n,\delta,\nu,K,[A]_{z',\delta/2,\delta},[B]_{z',\delta/2,\delta})$.
\end{theorem}

\begin{remark}
                                    \label{rem2.3}
It follows from the definition of $\hat U$ that $D_d u$ is continuous in $z'$ in Theorem \ref{thm1}, and is in $C_{z'}^{\delta/2,\delta}$ in Theorem \ref{thm2}. In other words, $D_d u$ is only discontinuous in the $x^d$ direction, which is a quite intuitive phenomenon from a physical point of view.
\end{remark}

\begin{remark}
The conditions on $f,g$ and $B$ in Theorems \ref{thm1} and \ref{thm2} can be relaxed. From the proofs below, it is easily seen that we only need $f$ to be in some weaker Morrey space (cf. \eqref{eq4.34}). For $g_\alpha,B^\alpha,\alpha=1,\cdots,d-1$, we only require the regularity with respect to the  $x^\alpha$-direction, in which the derivative is taken.
\end{remark}

\begin{remark}
                                        \label{rem5.1}
i) Suppose that $Q_1$ is divided into $M$ laminate sub-domains $\cD_1,\cD_2,\cdots,\cD_M$ by $M-1$ parallel hyperplanes with the common normal direction $(0,\cdots,0,1)$. Under the conditions of Theorem \ref{thm2}, we assume in addition that $A^{d\beta},\beta=1,\cdots,d$, $B^d$ and $g_d$ are in $C^{\delta/2,\delta}(\overline{\cD_i})$ for each $i=1,2,\cdots,M$, but may have jump discontinuities across these hyperplanes. Then we infer that
$$
D_d u=(A^{dd})^{-1}\left(\hat U+g_d-B^d u-\sum_{\beta=1}^{d-1}A^{d\beta}D_\beta u\right)
$$
is also in $C^{\delta/2,\delta}(\overline{\cD_i}\cap Q_{1-\epsilon})$ for any $i=1,\cdots,M$ and $\epsilon\in (0,1/2)$, with the $C^{\delta/2,\delta}$ norms independent of $M$. Similarly, in Theorem \ref{thm1}, if we additionally assume that $A^{d\beta},\beta=1,\cdots,d$, $B^d$ and $g_d$ are piecewise continuous, then $D_d u$ is also continuous in $\overline{\cD_i}\cap Q_{1-\epsilon}$.

ii) Our results can be extended to systems which model composite materials similar to those considered in \cite{LiVo,LiNi}. Let $T>0$, and $\Omega$ be a bounded domain in $\bR^d$ which consists $M$ disjoint subdomains $\Omega_1,\cdots,\Omega_M$ with $C^{1,\delta}$ boundaries. Suppose that any point belongs to the boundaries of at most two of the subdomains. We also assume that $A^{\alpha\beta}$, $B^\alpha$ and $g_\alpha$ are in $C^{\delta/2,\delta}((-T,0]\times \overline{\Omega_i})$ for each $i=1,2,\cdots,M$. For $\epsilon>0$, denote $$
\Omega_\epsilon=\{x\in \Omega:\text{dist}(x,\partial\Omega)>\epsilon\}.
$$
Let $u$ be a weak solution to \eqref{parabolic} in $(-T,0)\times \Omega$.
Then by the standard technique of locally flattening the boundaries, we can apply Theorem \ref{thm2} to obtain that
$$
u\in C^{(1+\delta)/2,1+\delta}((-T+\epsilon,0]\times (\Omega_\epsilon\cap \overline{\Omega_i}))
$$
for any $i=1,\cdots,M$ and $\epsilon\in (0,1/2)$. Similarly, if $\Omega_1,\cdots,\Omega_M$ have $C^{1,\text{Dini}}$ boundaries and $A^{\alpha\beta}$, $B^\alpha$, $g_\alpha$ are Dini continuous in $(-T,0]\times \overline{\Omega_i}$ for each $i=1,2,\cdots,M$, then from Theorem \ref{thm1} we get
$$
u\in C^{1/2,1}((-T+\epsilon,0]\times (\Omega_\epsilon\cap \overline{\Omega_i})),\quad i=1,\cdots,M,
$$
and $Du$ is continuous in $(-T+\epsilon,0]\times (\Omega_\epsilon\cap \overline{\Omega_i})$. However, the bounds of $u$ may also depends on the distances between different inclusions. Therefore, this method does not apply to the case when the boundaries of three subdomains touch at some point. We remark that the estimates obtained in \cite{LiVo, LiNi} (see also recently \cite{LiLi} and \cite{FKNN}) are independent of the distances between inclusions, under a restriction on the range of $\delta$. Estimates of this type are interesting from a physical point of view because, in the elliptic
model of extreme valued conductivity, the gradient blows up as the distance goes to zero.
\end{remark}

Next we state the results for elliptic systems \eqref{elliptic}, which follow immediately from Theorems \ref{thm1} and \ref{thm2} by viewing the solutions to the elliptic systems as a steady state solution to the corresponding parabolic systems.

\begin{corollary}
                            \label{cor1e}
Let $A\in C_{x'}^{\text{Dini}}$, $B\in C_{x'}^{\text{Dini}}$, $f\in L_\infty(B_1)$ and $g\in C_{x'}^{\text{Dini}}(B_1)$. Assume that $u$ is a weak solution to \eqref{elliptic} in $B_1$. Then we have
\begin{equation*}
|Du|_{0;B_{1/2}}\le N(I[\omega_{g,z'}](1)+|g|_{0;B_1}+|f|_{0;B_1}
+\|u\|_{L_2(B_1)}),
\end{equation*}
where $N=N(d,n,\nu,K,\omega_{A,z'},\omega_{B,z'})$. Moreover, $D_{x'}u$ and $\hat U$ are continuous in $\overline{B_{1/2}}$.
\end{corollary}

\begin{corollary}
                            \label{cor2e}
Let $\delta\in (0,1)$, $A\in C_{x'}^{\delta}$, $B\in C_{x'}^{\delta}$, $f\in L_\infty(B_1)$ and $g\in C_{x'}^{\delta}(B_1)$. Assume that $u$ is a weak solution to \eqref{elliptic} in $B_1$. Then we have
\begin{equation*}
|Du|_{0;B_{1/2}}+
[D_{x'}u]_{\delta;B_{1/2}}+[\hat U]_{\delta;B_{1/2}}\le N(|g|_{x',\delta;B_1}+|f|_{0;B_1}
+\|u\|_{L_2(B_1)}),
\end{equation*}
where $N=N(d,n,\delta,\nu,K,[A]_{x',\delta},[B]_{x',\delta})$.
\end{corollary}

Estimates in the same spirit as in Theorems \ref{thm1} and \ref{thm2} for non-divergence form equations are stated in Section \ref{sec6}.

\mysection{Some auxiliary estimates}
                            \label{sec3}

We will use the following property of Dini functions.
\begin{lemma}
                                                    \label{lem4.08}
Suppose that $\omega$ is a Dini function, and
\begin{equation}
                                    \label{eq4.22}
\tilde \omega(t):=\sum_{k=0}^\infty a^k\left(\omega(b^k t)\chi_{b^kt\le 1}+\omega(1)\chi_{b^kt> 1}\right),
\end{equation}
for some constants $a\in (0,1)$ and $b>1$.
Then $\tilde \omega$ is also a Dini function.
\end{lemma}
\begin{proof}
Let $\hat \omega(t)=\omega(t)$ for $t\le 1$ and $\hat \omega(t)=\omega(1)$ for $t>1$. Then $\hat \omega$ is bounded, uniformly continuous Dini function. Moreover,
$$
\tilde \omega(t)=\sum_{k=0}^\infty a^k\hat \omega(b^k t).
$$
Then it is easy to see that $\tilde \omega$ is increasing and continuous.
It follows from \eqref{eq4.22} that
$$
\tilde \omega(t)\le \sum_{k=0}^\infty a^k\omega(b^k t)\chi_{b^kt\le 1}+Nt^{\gamma}
$$
for some constants $N$ and $\gamma>0$ depending only on $a,b$ and $\omega_1$. By Fubini's theorem, we also get $\int_0^1 \tilde \omega(t)/t\,dt<\infty$. The lemma is proved.
\end{proof}

We will also need the following energy inequality and a variant of the parabolic Poincar\'e inequality.

\begin{lemma}
                                \label{lem3.6}
Let $B=\hat B=C=0$ and $R>0$. Suppose $v\in \cH^1_2(Q_{R})$ is a weak solution to the equation
\[
\left\{
  \begin{aligned}
    \cP v= \Div g+f \quad & \hbox{in $Q_{R}$;} \\
    v=0 \quad & \hbox{on $\partial_p Q_{R}$,}
  \end{aligned}
\right.
\]
where $f,g\in L_2(Q_R)$. Then we have
$$
\|Dv\|_{L_2(Q_R)}+R^{-1} \|v\|_{L_2(Q_R)}
\le N\|g\|_{L_2(Q_R)}+NR\|f\|_{L_2(Q_R)},
$$
where $N=N(d,n,\nu)>0$.
\end{lemma}

\begin{lemma}
                                        \label{paraPoin}
Let $p \in (1,\infty)$, $r \in (0,\infty)$ and $u \in C^\infty_{\textit{loc}}(\bR^{d+1})$.

i) Suppose $B=\hat B=C=0$, $g,f\in L_{p,\text{loc}}$, and $\cP u=\Div g+f$ in $Q_r$. Then
\begin{equation}
                                            \label{eq22.58}
\int_{Q_r}|u(t,x)-(u)_{Q_r}|^p\,dz\le Nr^{p}\int_{Q_r}(|Du|^p+|g|^p+r^p |f|^p)\,dz,
\end{equation}
where $N=N(d,\nu,p)>0$.

ii) There is a constant $N = N(d,p)$ such that
\begin{equation}
                \label{eq23.06a}
\begin{split}
\int_{Q_r}|Du(t,x)-(Du)_{Q_r}|^p\,dz&\le Nr^p\int_{Q_r}(|D^2u|^p+|u_t|^p)\,dz,\\
\int_{Q_r}|u(t,x)-(u)_{Q_r}-x^\beta(D_\beta u)_{Q_r}|^p\,dz&\le Nr^{2p}\int_{Q_r}(|D^2u|^p+|u_t|^p)\,dz.
\end{split}
\end{equation}
\end{lemma}
\begin{proof}
See, for instance, Lemma 3.1 and Lemma 3.2 of \cite{Krylov_2005}. We remark that by an approximation argument, \eqref{eq22.58} remains valid for $u\in \cH^1_{p,\text{loc}}$, and \eqref{eq23.06a} remains valid for $u\in W^{1,2}_{p,\text{loc}}$.
\end{proof}

In the remaining part of the section, we shall prove some local estimates, which are deduced from the results obtained in \cite{DK09}.
\begin{lemma}
                                            \label{lem3.1}
Let $p\in (1,\infty)$. Assume  $A^{\alpha\beta}$ are partially VMO in $z'$, $u\in \cH^1_p(Q_1)$ and
\begin{equation}
                                        \label{eq3.12}
\cP u=\Div g+f,
\end{equation}
in $Q_1$, where $f,g\in L_p(Q_1)$. Then there exists a constant $N=N(d,n,\nu,K,\omega_{A,z'},p)$ such that
\begin{equation*}
\|u\|_{\cH_p^1(Q_{1/2})}\le N(\|u\|_{L_p(Q_1)}+\|g\|_{L_p(Q_1)}+\|f\|_{L_p(Q_1)}).
\end{equation*}
\end{lemma}
\begin{proof}
The lemma follows from Theorem 2.2 of \cite{DK09} by a standard localization argument.
For the sake of completeness, we give a proof in the Appendix.
\end{proof}

By using the Sobolev imbedding theorem and a bootstrap argument, we get
\begin{corollary}
                                    \label{cor3.1}
Let $p,q\in (1,\infty)$. Assume $A^{\alpha\beta}$  are partially VMO in $z'$, $u\in C^\infty_{\text{loc}}$ satisfies \eqref{eq3.12} in $Q_1$, where $f,g\in L_q(Q_1)$. Then there exists a constant $N=N(d,n,\nu,K,\omega_{A,z'},p,q)$ such that
\begin{equation*}
\|u\|_{\cH^1_q(Q_{1/2})}\le N(\|u\|_{L_p(Q_1)}+\|g\|_{L_q(Q_1)}+\|f\|_{L_q(Q_1)}).
\end{equation*}
In particular, if $q>d+2$, it holds that
$$
|u|_{\gamma/2,\gamma;Q_{1/2}}\le N(\|u\|_{L_p(Q_1)}+\|g\|_{L_q(Q_1)}+\|f\|_{L_q(Q_1)}),
$$
where $\gamma=1-(d+2)/q$.
\end{corollary}

Next we consider systems with coefficients depending on $x^d$ alone. We denote
$$
\cP_0 u=-u_t+D_\alpha(A^{\alpha\beta}(x^d)D_\beta u),
$$
and
\begin{equation*}							
U:=A^{d\beta}D_{\beta}u,
\quad
\text{i.e.,}
\quad
U^i=A^{d\beta}_{ij}D_{\beta}u^j,
\quad
i = 1, \cdots, n.
\end{equation*}
\begin{lemma}
                                    \label{lem3.2}
Let $p\in (1,\infty)$. Assume $u\in C^\infty_{\text{loc}}$ satisfies $\cP_0 u=0$ in $Q_1$. Then for any nonnegative integers $i,j$ such that $i+j\ge 1$ and any $q\in (1,\infty)$, there exists a constant $N=N(d,n,p,q,i,j,\nu)$ such that
\begin{equation}
                                    \label{eq3.23}
\|D_t^i D_{x'}^j u\|_{\cH^1_q(Q_{1/2})}\le N\|Du\|_{L_p(Q_1)}.
\end{equation}
For any $\gamma\in (0,1)$, we also have
\begin{equation}
                                    \label{eq3.48}
|D_t^i D_{x'}^j u|_{\gamma/2,\gamma;Q_{1/2}}\le N\|Du\|_{L_p(Q_1)}.
\end{equation}
Moreover,
\begin{align}
                                    \label{eq3.53}
\|D_t^i D_{x'}^j U\|_{\cH^1_q(Q_{1/2})}&\le N\|Du\|_{L_p(Q_1)},\\
                                        \label{eq3.53b}
|D_t^i D_{x'}^j U|_{\gamma/2,\gamma;Q_{1/2}}&\le N\|Du\|_{L_p(Q_1)}.
\end{align}
\end{lemma}
\begin{proof}
Note that $\cP_0(D_t^i D_{x'}^j u)=0$ in $Q_1$. Thanks to Lemma \ref{lem3.1} and Corollary \ref{cor3.1}, to prove \eqref{eq3.23} it suffices to show that for any $1/2\le r<R\le 1$,
\begin{equation}
                                        \label{eq22.22}
\|u_t\|_{L_2(Q_{r})}\le N(d,n,\nu,r,R)\|Du\|_{L_2(Q_{R})}.
\end{equation}
Indeed, if $i\ge 1$, by Lemma \ref{lem3.1} and \eqref{eq22.22}, for any $1/2\le r_1<r_2<r_3\le1$,
$$
\|D_t^i D_{x'}^j u\|_{\cH^{1}_q(Q_{r_1})}\le N\|D_t^i D_{x'}^j u\|_{L_2(Q_{r_2})}
\le N\|D_t^{i-1} D_{x'}^j u\|_{\cH^{1}_2(Q_{r_3})}.
$$
Similarly, if $j\ge 1$, we have
$$
\|D_t^i D_{x'}^j u\|_{\cH^{1}_q(Q_{r_1})}\le N\|D_t^i D_{x'}^j u\|_{L_2(Q_{r_2})}
\le N\|D_t^{i} D_{x'}^{j-1} u\|_{\cH^{1}_2(Q_{r_3})}.
$$
Repeating this procedure to reduce $i$ and $j$, we reach
$$
\|D_t^i D_{x'}^j u\|_{\cH^{1,2}_q(Q_{1/2})}
\le N\|u_t\|_{L_2(Q_{3/4})}+N\|Du\|_{L_2(Q_{3/4})}\le N\|Du\|_{L_2(Q_{1})}.
$$
In the last inequality, we used \eqref{eq22.22}.
For the proof of \eqref{eq22.22}, see for instance Lemma 3.3 of \cite{DK09}. Inequality \eqref{eq3.48} is deduced from \eqref{eq3.23} by the parabolic Sobolev imbedding theorem.

To prove \eqref{eq3.53} and \eqref{eq3.53b}, one only needs to observe that in $Q_1$,
$$
D_dU = u_t - \sum_{\alpha=1}^{d-1}\sum_{\beta=1}^d D_{\alpha}(A^{\alpha\beta}D_{\beta}u)
= u_t -\sum_{\alpha=1}^{d-1}\sum_{\beta=1}^d A^{\alpha\beta}D_{\alpha\beta}u,
$$
and
$$
D_{x'}U=\sum_{\beta=1}^d A^{d\beta}D_{x'}D_{\beta}u,\quad
D_{t}U=\sum_{\beta=1}^d A^{d\beta}D_tD_{\beta}u.
$$
The $L_p$ norms of the right-hand sides are bounded by $\|Du\|_{L_p}$ due to \eqref{eq3.23}.
\end{proof}

\mysection{Systems with partially Dini coefficients}
                            \label{sec4}

This section is devoted to the proof of Theorem \ref{thm1}. The following lemma reduces the estimate of $[u]_{1/2,1}$ to the estimate of $|Du|_0$.

\begin{lemma}
                                \label{lem4.9}
Let $u$ be a weak solution to \eqref{parabolic} in $Q_1$. Suppose $|u|_{0;Q_{1/2}}<\infty$ and $|Du|_{0;Q_{1/2}}<\infty$. Then we have
\begin{equation}
                                            \label{eq1.48}
[u]_{1/2,1;Q_{1/4}}\le N(|u|_{0;Q_{1/2}}+|Du|_{0;Q_{1/2}}+|f|_{0;Q_{1/2}}+|g|_{0;Q_{1/2}}).
\end{equation}
\end{lemma}
\begin{proof}
Rewrite \eqref{parabolic} as
$$
-u_t+D_\alpha(A^{\alpha\beta}D_\beta u)=D_\alpha(g_\alpha-B^{\alpha}u)+f-\hat B^{\alpha}D_\alpha u- C u:=D_\alpha \fg_\alpha+\ff,
$$
where
\begin{equation}
                    \label{eq328.20.58}
\fg_\alpha=g_\alpha-B^{\alpha}u,\quad
\ff=f-\hat B^{\alpha}D_\alpha u- C u.
\end{equation}
We fix $z_0\in \overline{Q_{1/4}}$ and take $r\in (0,1/4)$. By Lemma \ref{paraPoin}, we have
$$
\int_{Q_r(z_0)}|u-(u)_{Q_r(z_0)}|^2\,dz
\le Nr^2\int_{Q_r(z_0)}\left(|Du|^2+|\fg|^2+r^2|\ff|^2\right)
$$
$$
\le Nr^{n+4}\left(|Du|_{0;Q_{1/2}}+|u|_{0;Q_{1/2}}
+|g|_{0;Q_{1/2}}+|f|_{0;Q_{1/2}}\right)^2.
$$
The inequality \eqref{eq1.48} then follows from Campanato's characterization of H\"older continuous functions; see, for instance, \cite[Lemma 4.3]{Lieberman}.
\end{proof}

Now we are ready to prove Theorem \ref{thm1}. We shall divide the proof into two steps.

{\em Step 1: Estimate of $|Du|_{0}$.}
We first assume that all the coefficients and data are smooth, so that by the classical theory $|Du|_{0;Q_{3/4}}<\infty$. We take $0<\gamma_1<\gamma<1$.
Fix a point $z_0\in Q_{3/4}$, and take $0<r<R\le (3/4-|x_0|)/2$. Now take $z_1'\in Q'_{R}(z_0')$ and denote
$$
\cP_{z_1'} u:=-u_t+D_\alpha(A^{\alpha\beta}(z_1',x^d)D_\beta u).
$$
Recall the definitions of $\fg$ and $\ff$ in \eqref{eq328.20.58}.
Then we have
$$
\cP_{z_1'} u=\Div (\fg+m)+\ff,
$$
where
$$
m_\alpha^i(z)=(A^{\alpha\beta}_{ij}(z_1',x^d)-A^{\alpha\beta}_{ij}(z))
D_\beta u^j.
$$
Let $$
u_0(x^d)=\int_{-1}^{x^d}(A^{dd}(z_1',s))^{-1}\fg_d(z_1',s)\,ds,
\quad u_e=u-u_0.
$$
Clearly,
\begin{equation*}
\cP_{z_1'} u_e=\Div (\fg(z)-\fg(z_1',x^d)+m(z))+\ff.
\end{equation*}

Let $v$ be a weak solution to the equation
\[
\left\{
  \begin{aligned}
    \cP_{z_1'} v= \Div (\fg(z)-\fg(z_0',x^d)+m(z))+\ff \quad & \hbox{in $Q_{R}(z_0)$;} \\
    v=0 \quad & \hbox{on $\partial_p Q_{R}(z_0)$.}
  \end{aligned}
\right.
\]
By Lemma \ref{lem3.6}, we get
\begin{align}
&\|Dv\|_{L_2(Q_{R}(z_0))}\nonumber\\
&\le N\|\fg(z)-\fg(z_1',x^d)+m\|_{L_2(Q_{R}(z_0))}+NR\|\ff\|_{L_2(Q_{R}(z_0))}\nonumber\\
&\le N\|\fg(z)-\fg(z_1',x^d)+m\|_{L_2(Q_{R}(z_0))}+N
|\ff|_{0;Q_{R}(z_0)}R^{d/2+1+\gamma_1}.
                                \label{eq4.34}
\end{align}

Let $w=u_e-v$. Then $w$ satisfies $\cP_{z_1'} w = 0$ in $Q_{R}(z_0)$. Denote
\begin{equation}
                                \label{eq21.11}
W:=
A^{d\beta}(z_1',x^d)D_{\beta}w.
\end{equation}
It follows from Lemma \ref{lem3.2} and a suitable scaling that
\begin{align}
&\int_{Q_r(z_0)}|D_{x'}w-(D_{x'}w)_{Q_r(z_0)}|^2+|W-(W)_{Q_r(z_0)}|^2\,dz\nonumber\\
&\,\le N(r/R)^{d+2+2\gamma}
\int_{Q_R(z_0)}|Dw|^2\,dz\nonumber\\
&\,\le N(r/R)^{d+2+2\gamma}
\int_{Q_R(z_0)}|D_{x'}w|^2+|W|^2\,dz.
                                    \label{eq23.33}
\end{align}
In the last inequality, we also used the nondegeneracy of $A^{dd}$.
Define
$$
h(x^d)=\int_0^{x^d}(A^{dd}(z_1',s))^{-1}
\left((W)_{Q_R(z_0)}-\sum_{\beta=1}^{d-1}A^{d\beta}(z_1',\cdot)
(D_{\beta}w)_{Q_R(z_0)}\right)\,ds,
$$
and
$$
\tilde w:=w-\sum_{\beta=1}^{d-1}x^\beta (D_{\beta}w)_{Q_R(z_0)}-h(x^d).
$$
We define $\tilde W$ and $U_e$ as in \eqref{eq21.11} with $\tilde w$ and $u_e$ in place of $w$ respectively. Then,
$$
D_{x'}\tilde w=D_{x'}w-(D_{x'}w)_{Q_R(z_0)},\quad \tilde W=W-(W)_{Q_R(z_0)}.
$$
A direct calculation shows that $\tilde w$ also satisfies $\cP_{z_1'} \tilde w = 0$ in $Q_{R}(z_0)$. We substitute $w$ and $W$ in \eqref{eq23.33} by $\tilde w$ and $\tilde W$ to get
\begin{align}
&\int_{Q_r(z_0)}|D_{x'}w-(D_{x'}w)_{Q_r(z_0)}|^2+|W-(W)_{Q_r(z_0)}|^2
\,dz\nonumber\\
&\,\le N(r/R)^{d+2+2\gamma}
\int_{Q_R(z_0)}|D_{x'}w-(D_{x'}w)_{Q_R(z_0)}|^2+|W-(W)_{Q_R(z_0)}|^2\,dz.
                                    \label{eq5.35}
\end{align}
We combine \eqref{eq4.34} with \eqref{eq5.35} and use the triangle inequality to obtain
\begin{align}
&\int_{Q_r(z_0)}|D_{x'}u_e-(D_{x'}u_e)_{Q_r(z_0)}|^2+|U_e-(U_e)_{Q_r(z_0)}|^2\,dz\nonumber\\
&\,\le N_1(r/R)^{d+2+2\gamma}
\int_{Q_R(z_0)}|D_{x'}u_e-(D_{x'}u_e)_{Q_R(z_0)}|^2+|U_e-(U_e)_{Q_R(z_0)}|^2\,dz\nonumber\\
&\,\,+N\|\fg(z)-\fg(z_1',x^d)+m\|^2_{L_2(Q_{R}(z_0))}+N
|\ff|^2_{0;Q_{R}(z_0)}R^{d+2+2\gamma_1}.
                                    \label{eq5.36}
\end{align}
where $N_1=N_1(d,n,\nu)$. It is not very convenient to use \eqref{eq5.36} since both $u_e$ and $U_e$ depend on $z_0$.
However, by the definition of $u_0$, $u_e$ and $U_e$,
$$
D_{x'}u_e=D_{x'}u,\quad U_e=
A^{d\beta}(z_1',x^d)D_\beta u-\fg_d(z_1',x^d)
$$
Recall that
\begin{equation}
                                        \label{eq22.19}
\hat U=
A^{d\beta}D_\beta u-\fg_d(z),
\end{equation}
which is independent of the choice of $z_0$.
Clearly, in $Q_R(z_0)$
$$
|\hat U(z)-U_e(z)|\le N|A(z)-A(z_1',x^d)||Du|_{0;Q_{R}(z_0)}+|\fg_d(z)-\fg_d(z_1',x^d)|.
$$
Thus, coming back to \eqref{eq5.36} we get
\begin{align}
&\int_{Q_r(z_0)}|D_{x'}u-(D_{x'}u)_{Q_r(z_0)}|^2+|\hat U-(\hat U)_{Q_r(z_0)}|^2\,dz\nonumber\\
&\,\le N_1(r/R)^{d+2+2\gamma}
\int_{Q_R(z_0)}|D_{x'}u-(D_{x'}u)_{Q_R(z_0)}|^2+|\hat U-(\hat U)_{Q_R(z_0)}|^2\,dz\nonumber\\
&\,\,+N\|\fg(z)-\fg(z_1',x^d)\|^2_{L_2(Q_{R}(z_0))}
+N\|A(z)-A(z_1',x^d)\|^2_{L_2(Q_{R}(z_0))}
|Du|^2_{0;Q_{R}(z_0)}\nonumber\\
&\,\,+N|\ff|^2_{0;Q_{R}(z_0)}R^{d+2+2\gamma_1}.
                                    \label{eq5.36bb}
\end{align}
Now we take the average of both sides of \eqref{eq5.36bb} with respect to $z_1'\in Q_R'(z_0)$, and use the definitions of $\fg$ and $\ff$ to obtain
\begin{align}
&\int_{Q_r(z_0)}|D_{x'}u-(D_{x'}u)_{Q_r(z_0)}|^2+|\hat U-(\hat U)_{Q_r(z_0)}|^2\,dz\nonumber\\
&\,\le N_1(r/R)^{d+2+2\gamma}
\int_{Q_R(z_0)}|D_{x'}u-(D_{x'}u)_{Q_R(z_0)}|^2+|\hat U-(\hat U)_{Q_R(z_0)}|^2\,dz\nonumber\\
&\,\,+N\left(\omega_{g,z'}(R)+\omega_{B,z'}(R)|u|_{0;Q_{3/4}}
+\omega_{A,z'}(R)|Du|_{0;Q_{R}(z_0)}\right)^2R^{d+2}\nonumber\\
&\,\,+N\left([u]_{\gamma_1/2,\gamma_1;Q_{3/4}}+|\ff|_{0;Q_{R}(z_0)}\right)^2 R^{d+2+2\gamma_1}.
                                    \label{eq4.45}
\end{align}

Set $r=\tau R$ for some $\tau\in (0,1)$ to be chosen later, and denote
$$
\phi_r(z_0)=\dashint_{Q_r(z_0)}|D_{x'}u-(D_{x'}u)_{Q_r(z_0)}|^2
+|\hat U-(\hat U)_{Q_r(z_0)}|^2\,dz.
$$
It follows from \eqref{eq4.45} that
$$
\phi_{\tau R}(z_0)\le N_1\tau^{2\gamma}\phi_{R}(z_0)+N\left(
\omega_{g,z'}^2(R)+\omega_{B,z'}^2(R)|u|^2_{0;Q_{3/4}}
\right)\tau^{-d-2}
$$
$$
+N\omega_{A,z'}^2(R)|Du|^2_{0;Q_{R}(z_0)}\tau^{-d-2}
+N\left([u]^2_{\gamma_1/2,\gamma_1;Q_{3/4}}+|\ff|^2_{0;Q_{R}(z_0)}\right) \tau^{-d-2}R^{2\gamma_1}.
$$
We fix $\tau=\tau(d,n,\nu)<1$ sufficiently small such that $N_1\tau^{2\gamma}\le 1/2$. By an iteration, we obtain,
$$
\phi_{\tau^k R}(z_0)\le 2^{-k}\phi_{R}(z_0)+N\sum_{j=1}^{k}2^{-j}\Big(\omega_{g,z'}^2(\tau^{k-j}R)
+\omega_{B,z'}^2(\tau^{k-j}R)|u|^2_{0;Q_{3/4}}
$$
\begin{equation}
                                \label{eq5.48}
+\omega_{A,z'}^2(\tau^{k-j}R)|Du|^2_{0;Q_{R}(z_0)}\Big)
+N\left([u]_{\gamma_1/2,\gamma_1;Q_{3/4}}+|\ff|_{0;Q_{R}(z_0)}\right)^2 \sum_{j=1}^{k}2^{-j}(\tau^{k-j}R)^{2\gamma_1}.
\end{equation}
We define
$$
\psi_r(z_0)=\dashint_{Q_r(z_0)}|D_{x'}u|+|\hat U|\,dz.
$$
By the triangle inequality and H\"older's inequality,
\begin{align}
|\psi_{\tau r}(z_0)-\psi_r(z_0)|&\le \dashint_{Q_{\tau r}(z_0)}|D_{x'}u-(D_{x'}u)_{Q_r(z_0)}|+|\hat U-(\hat U)_{Q_r(z_0)}|\,dz\nonumber\\
                                \label{eq1.46}
&\le N(\phi_r(z_0))^{1/2}.
\end{align}
Combining \eqref{eq5.48} and \eqref{eq1.46}, we deduce
\begin{multline*}
|\psi_{\tau^k R}(z_0)-\psi_{\tau^{k-1} R}(z_0)|\le
N2^{-k/2}\phi_{R}^{1/2}(z_0)\\
+N\sum_{j=1}^{k}2^{-j/2}
\Big(\omega_{g,z'}(\tau^{k-j}R)
+\omega_{B,z'}(\tau^{k-j}R)|u|_{0;Q_{3/4}}
+\omega_{A,z'}(\tau^{k-j}R)|Du|_{0;Q_{R}(z_0)}\Big)\\
+N\left([u]_{\gamma_1/2,\gamma_1;Q_{3/4}}+|\ff|_{0;Q_{R}(z_0)}\right) \sum_{j=1}^{k}2^{-j/2}(\tau^{k-j}R)^{\gamma_1}.
\end{multline*}
Summing the inequality above in $k$ gives
$$
\psi_{\tau^k R}(z_0)\le \psi_{R}(z_0)+
N\phi_{R}^{1/2}(z_0)+N\sum_{j=1}^{k}\left(\omega_{g,z'}(\tau^{k-j}R)
+\omega_{B,z'}(\tau^{k-j}R)|u|_{0;Q_{3/4}}\right)
$$
$$
+N\sum_{j=1}^k\omega_{A,z'}(\tau^{k-j}R)|Du|_{0;Q_{R}(z_0)}
+N\left([u]_{\gamma_1/2,\gamma_1;Q_{3/4}}+|\ff|_{0;Q_{R}(z_0)}\right) R^{\gamma_1},
$$
where $N=N(d,n,\nu)$. To estimate the summations on the right-hand side, we recall
$$
I[\omega](r)=\int_0^r \omega(s)/s\,ds.
$$
It is easy to see that for any Dini function $\omega$ and $k\ge 1$,
$$
\sum_{j=1}^{k}\omega(\tau^{k-j}R)\le NI[\omega](\tau^{-1} R).
$$
Therefore, we deduce
\begin{multline}
\psi_{\tau^k R}(z_0)\le \psi_{R}(z_0)+
N\phi_{R}^{1/2}(z_0)+NI[\omega_{g,z'}](\tau^{-1} R)
+NI[\omega_{B,z'}](\tau^{-1} R)|u|_{0;Q_{3/4}}\\
                            \label{eq19.54}
+NI[\omega_{A,z'}](\tau^{-1} R)|Du|_{0;Q_R(z_0)}
+N\left([u]_{\gamma_1/2,\gamma_1;Q_{3/4}}+|\ff|_{0;Q_{R}(z_0)}\right) R^{\gamma_1}.
\end{multline}
It follows from \eqref{eq19.54} and the definition of $\ff$ that
\begin{align}
&(|Du|)_{Q_{\tau^k R}(z_0)}\le \psi_{R}(z_0)+
N\phi_{R}^{1/2}(z_0)+NI[\omega_{g,z'}](\tau^{-1} R)\nonumber\\
&\qquad+NI[\omega_{B,z'}](\tau^{-1} R)|u|_{0;Q_{3/4}}
+N(|g|_{0;Q_{3/4}}+|u|_{0;Q_{3/4}}+[u]_{\gamma_1/2,\gamma_1;Q_{3/4}})\nonumber\\
                            \label{eq20.00}
&\qquad+N_2(I[\omega_{A,z'}](\tau^{-1} R)+R^{\gamma_1})|Du|_{0;Q_R(z_0)}
+N|f|_{0;Q_{R}(z_0)} R^{\gamma_1},
\end{align}
where $N_2=N_2(d,n,\nu)>0$. We choose $R_0=R_0(d,n,\nu,\omega_{A,z'})\in (0,\tau/4))$ such that
$$
N_2(I[\omega_{A,z'}](\tau^{-1} R_0)+R_0^{\gamma_1})<2^{-d-2}.
$$
Now we confine $z_0$ to $Q^{(\ell)}:=Q_{3/4-2^{-\ell-1}R_0},\ell=1,2,...$ and set $R=2^{-\ell-2}R_0$. Since $k$ and $z_0\in Q^{(\ell)}$ in \eqref{eq20.00} are arbitrary, due to the Lebesgue lemma, we obtain
$$
|Du|_{0;Q^{(\ell)}}\le N_3 2^{\ell(d+2)/2}\left(\|Du\|_{L_2(Q_1)}+|g|_{0;Q_{3/4}}+I[\omega_{g,z'}](1)+|u|_{0;Q_{3/4}}
\right)
$$
\begin{equation}
                                \label{eq2.40b}
+N_3 2^{\ell(d+2)/2}\left(|f|_{0;Q_{R}(z_0)}+[u]_{\gamma_1/2,\gamma_1;Q_{3/4}}\right)
+2^{-d-2}|Du|_{0;Q^{(\ell+1)}},
\end{equation}
where $N_3=N_3(d,n,\nu,\omega_{A,z'},\omega_{B,z'})>0$. Multiplying both sides of \eqref{eq2.40b} by $2^{-\ell(d+2)}$ and summing in $\ell=1,2,...$ yield
\begin{multline*}
\sum_{\ell=1}^\infty2^{-\ell(d+2)}|Du|_{0;Q^{(\ell)}}\le N_3 \left(\|Du\|_{L_2(Q_1)}+|g|_{0;Q_{3/4}}+I[\omega_{g,z'}](1)+|u|_{0;Q_{3/4}}\right)\\
+ N_3\left(|f|_{0;Q_{R}(z_0)}+[u]_{\gamma_1/2,\gamma_1;Q_{3/4}}\right)
+\sum_{\ell=1}^\infty 2^{-(\ell+1)(d+2)}|Du|_{0;Q^{(\ell+1)}}.
\end{multline*}
Since we assume $|Du|_{0;Q_{3/4}}<\infty$, it follows from the inequality above by absorbing the summation on the right-hand side into the left-hand side that
$$
|Du|_{0;Q^{(1)}}\le
N_3 \left(\|Du\|_{L_2(Q_1)}+|g|_{0;Q_{3/4}}+I[\omega_{g,z'}](1)+|u|_{0;Q_{3/4}}\right)
$$
$$
+N_3\left(|f|_{0;Q_{R}(z_0)}+[u]_{\gamma_1/2,\gamma_1;Q_{3/4}}\right).
$$
To finish the proof of \eqref{eq4.30}, it suffices to use Corollary \ref{cor3.1}.

Now we remove the smoothness assumption on the coefficients and data by using a standard approximation argument, which we sketched below for the completeness. Let $A^{\alpha\beta}_{(m)},m=1,2,...$ be the mollifications of $A^{\alpha\beta}$. Similarly, we define $B^\alpha_{(m)}$,  $\hat B^\alpha_{(m)}$, $C^\alpha_{(m)}$, $g_{(m)}$ and $f_{(m)}$. We know that $A^{\alpha\beta}_{(m)},m=1,2,...$ are uniformly elliptic with the same ellipticity constant $\nu$, and as $m\to \infty$,
\begin{align*}
\left(f_{(m)},g_{(m)}\right)&\to (f,g)\,\,\text{in }  L_2(Q_1),\\
\left(A^{\alpha\beta}_{(m)},B^\alpha_{(m)},\hat B^\alpha_{(m)}\right)&\to \left(A^{\alpha\beta},B^\alpha,\hat B^\alpha\right)\,\,\text{a.e.}.
\end{align*}
Moreover,
$$
\omega_{A_{(m)},z'}\le \omega_{A,z'},\quad \omega_{B_{(m)},z'}\le \omega_{B,z'}.
$$
Let $\cP_m$ be the parabolic operator with coefficients $A^{\alpha\beta}_{(m)}, B^\alpha_{(m)}$ and $C_{(m)}$ instead of $A^{\alpha\beta}, B^\alpha$ and $C$.
Let $v_m\in \cH^1_2(Q_1)$ be the weak solution to the equation
\begin{multline}
\cP_m v_m= \Div(g-g_{(m)})+f-f_{(m)}
    +D_\alpha((A^{\alpha\beta}_{(m)}-A^{\alpha\beta})D_\beta u\\
                                            \label{eq21.23}
+(B^\alpha_{(m)}-B^\alpha)u)+(\hat B^\alpha_{(m)}-\hat B^\alpha)D_\alpha u+
      (\hat C_{(m)}-\hat C)u
\end{multline}
in $Q_1$ with the zero initial-boundary condition $v_m=0$ on the parabolic boundary of $Q_1$. Note that the right-hand side of \eqref{eq21.23} goes to zero in $L_2$ as $m\to \infty$. It follows from the $L_2$ estimate that $\|v_m\|_{L_2(Q_1)}+\|Dv_m\|_{L_2(Q_1)}\to 0$ as $m\to \infty$, which further implies that there is subsequence, which is still denoted by $\{v_m\}$, such that $|Dv_m|+|v_m|\to 0$ a.e. in $Q_1$. Let $u_m=u-v_m$. It is clear that $u_m$ is a weak solution to the equation
$$
\cP_m u_m=\Div g_{(m)}+f_{(m)}
$$
in $Q_1$. By the classical parabolic theory, $u_m$ is infinitely differentiable in $Q_1$. Thus, by the proof above and by slightly shrinking the domain, we obtain a uniform estimate
\begin{align*}
|Du_m|_{0;Q_{1/2}}&\le N(I[\omega_{g,z'}](1)+|g_{(m)}|_{0;Q_{5/6}}+|f_{(m)}|_{0;Q_{5/6}}
+\|u_m\|_{L_2(Q_{5/6})})\\
&\le N(I[\omega_{g,z'}](1)+|g|_{0;Q_1}+|f|_{0;Q_1}
+\|u_m\|_{L_2(Q_1)}).
\end{align*}
Since $Du_m\to Du$ a.e. and $u_m\to u$ in $L_2$ in $Q_1$, by taking the limit in the above inequality  we get \eqref{eq4.30}.

{\em Step 2: Continuity of $D_{x'}u$ and $\hat U$.} Next, we prove the second claim of the theorem.
Fix a point $z_0\in \overline{Q_{1/4}}$ and take $0<r<R\le 1/4$.

We define $\omega(R)=\omega_{g,z'}(R)+\omega_{B,z'}(R)+\omega_{A,z'}(R)+R^{\gamma}$, which is a Dini function. By the estimates on $u$ and $Du$ which have already been established, we get from \eqref{eq4.45} that
\begin{align}
&\dashint_{Q_r(z_0)}|D_{x'}u-(D_{x'}u)_{Q_r(z_0)}|^2+|\hat U-(\hat U)_{Q_r(z_0)}|^2\,dz\nonumber\\
&\,\le N_1(r/R)^{2\gamma}
\int_{Q_R(z_0)}|D_{x'}u-(D_{x'}u)_{Q_R(z_0)}|^2+|\hat U-(\hat U)_{Q_R(z_0)}|^2\,dz\nonumber\\
&\,\,+N_4\omega(R)(R/r)^{d+2},
                                    \label{eq4.46}
\end{align}
where $N_4$ is independent of $r,R$ and $z_0$.
Again we take $r=\tau R$ with $\tau=\tau(d,n,\nu)$ sufficiently small such that $N_1\tau^{2\gamma}<1/2$. By an iteration, we obtain from \eqref{eq4.46}
\begin{align*}
&\dashint_{Q_{\tau^kR}(z_0)}|D_{x'}u-(D_{x'}u)_{Q_r(z_0)}|^2+|\hat U-(\hat U)_{Q_r(z_0)}|^2\,dz\nonumber\\
&\,\le 2^{-k}
\dashint_{Q_R(z_0)}|D_{x'}u-(D_{x'}u)_{Q_R(z_0)}|^2+|\hat U-(\hat U)_{Q_R(z_0)}|^2\,dz\nonumber\\
&\,\,+N_4\sum_{j=1}^k 2^{-j}\omega(\tau^{k-j}R).
\end{align*}
We set $R=1/4$ and use Lemma \ref{lem4.08} to conclude that for $r\in (0,1/4)$
$$
\dashint_{Q_{r}(z_0)}|D_{x'}u-(D_{x'}u)_{Q_r(z_0)}|^2+|\hat U-(\hat U)_{Q_r(z_0)}|^2\,dz
\le \tilde \omega(r),
$$
where $\tilde \omega$ is a Dini function. It then follows from (6.3) of \cite{Sperner} that $D_{x'}u$ and $\hat U$ are uniformly continuous in $\overline{Q_{1/4}}$. By using a dilation and covering argument, the continuity of $D_{x'}u$ and $\hat U$ in $\overline{Q_{1/2}}$ follows.

\mysection{Systems with partially H\"older coefficients}
                        \label{sec5}
We prove Theorem \ref{thm2} in this section.
Thanks to Theorem \ref{thm1}, we only need to estimate the last three terms on the left-hand side of \eqref{eq4.20}. We estimate them separately by using different methods.

\subsection{Estimate of $D_{x'}u$ and $\hat U$}
We fix a point $z_0\in Q_{1/4}$ and take $0<r<R\le 1/4$. Following the proof of Theorem \ref{thm1}, we take $\gamma_1=\delta$ and $\gamma\in (\delta,1)$. Note that under the conditions of Theorem \ref{thm2}, we have
$$
\omega_{g,z'}(R)\le[g]_{z',\delta/2,\delta}R^\delta,
$$
and similar inequalities for $\omega_{A,z'}(R)$ and $\omega_{B,z'}(R)$. Owing to \eqref{eq4.45}, we get
\begin{align}
&\int_{Q_r(z_0)}|D_{x'}u-(D_{x'}u)_{Q_r(z_0)}|^2+|\hat U-(\hat U)_{Q_r(z_0)}|^2\,dz\nonumber\\
&\,\le N_1(r/R)^{d+2+2\gamma}
\int_{Q_R(z_0)}|D_{x'}u-(D_{x'}u)_{Q_R(z_0)}|^2+|\hat U-(\hat U)_{Q_R(z_0)}|^2\,dz\nonumber\\
&\,\,+N_2R^{d+2+2\delta}\left([g]_{z',\delta/2,\delta}+|\ff|_{0;Q_{R}(z_0)}
+|u|_{\delta/2,\delta;Q_{3/4}}
+|Du|_{0;Q_{R}(z_0)}\right)^2,
                                    \label{eq11.33}
\end{align}
where $N_1$ only depends on $d,n$ and $\nu$, and $N_2$ also depends on the $C_{z'}^{\delta/2,\delta}$ semi-norms of $A$ and $B$.

Since \eqref{eq11.33} holds for any $0<r<R\le 1/4$ and $\delta<\gamma$, by a well-known iteration argument
(see e.g., \cite[Lemma 2.1, p. 86]{Giaq83}),
\begin{align}
&\int_{Q_r(z_0)}|D_{x'}u-(D_{x'}u)_{Q_r(z_0)}|^2+|\hat U-(\hat U)_{Q_r(z_0)}|^2\,dz\nonumber\\
&\,\le N_1r^{d+2+2\delta}
\int_{Q_{1/4}(z_0)}|D_{x'}u-(D_{x'}u)_{Q_{1/4}(z_0)}|^2+|\hat U-(\hat U)_{Q_{1/4}(z_0)}|^2\,dz\nonumber\\
&\,\,+N_2r^{d+2+2\delta}\left([g]_{z',\delta/2,\delta;Q_1}+|\ff|_{0;Q_{R}(z_0)}
+|u|_{\delta/2,\delta;Q_{3/4}}
+|Du|_{0;Q_{R}(z_0)}\right)^2.
                                    \label{eq11.39}
\end{align}
By the definition of $\ff$ and $U_e$, we get from \eqref{eq11.39}, Theorem \ref{thm1}, Corollary \ref{cor3.1} that
\begin{align}
&\int_{Q_r(z_0)}|D_{x'}u-(D_{x'}u)_{Q_r(z_0)}|^2+|\hat U-(\hat U)_{Q_r(z_0)}|^2\,dz\nonumber\\
&\,\le N_2r^{d+2+2\delta}\left(|g|_{z',\delta/2,\delta;Q_1}+|f|_{0;Q_1}
+\|u\|_{L_2(Q_1)}\right)^2.
                                    \label{eq11.45}
\end{align}
Since \eqref{eq11.45} holds for any $r\in (0,1/4)$ and $z_0\in Q_{1/4}$, by Campanato's characterization of H\"older continuous functions, we obtain
\begin{equation*}
[D_{x'}u]_{\delta/2,\delta;Q_{1/4}}+[\hat U]_{\delta/2,\delta;Q_{1/4}}\le N_2\left(|g|_{z',\delta/2,\delta;Q_1}+|f|_{0;Q_1}
+\|u\|_{L_2(Q_1)}\right).
\end{equation*}
This together with a dilation and covering argument gives
\begin{equation}
[D_{x'}u]_{\delta/2,\delta;Q_{1/2}}+[\hat U]_{\delta/2,\delta;Q_{1/2}}\le N_2\left(|g|_{z',\delta/2,\delta;Q_1}+|f|_{0;Q_1}
+\|u\|_{L_2(Q_1)}\right).
                                    \label{eq11.55}
\end{equation}

\subsection{Estimate of $\ip{u}_{1+\delta}$}
We estimate $\ip{u}_{1+\delta}$ by modifying the argument in \cite{DongSeickK09}, which in turn is based on an idea by M. Safonov. The argument in \cite{DongSeickK09} uses the maximum principle, which is unavailable for systems. Here we use some estimates established in Section \ref{sec3} instead.

We denote by $\hat\bP_1$ the set of all functions $p$ on $\bR^{d+1}$ of the form
\[
p(z)=p(t,x^1,x^d)=\sum_{i=1}^{d-1} \alpha^i(x^d) x^i + \beta(x^d).
\]
Then we define the first-order partial Taylor's polynomial with respect to $z'=(t,x')$ of a function $v$ on $\bR^{d+1}$ at a point $z_0'=(t_0,x_0')$ as
\[
\hT^1_{z_0'}v(z',x^d):=v(z_0',x^d)+\sum_{i=1}^{d-1}  D_i v(z_0',x^d) (x^i-x_0^i).
\]
Let
$$
\zeta(z')=\zeta(t,x^1,\ldots,x^{d-1})=\eta(t+1)\prod_{i=1}^{d-1}\eta(x^i),
$$
where $\eta$ is a smooth even function on $\bR$ with a compact support in $(-1,1)$ satisfying
$$
\int_{\bR} \eta(t)\,dt=1,\quad \int_{\bR} t^2\eta(t)\,dt=0.
$$
For $\epsilon>0$ let $\zeta_\epsilon(t,x')=\epsilon^{-d-1}\zeta(\epsilon^{-2}t,\epsilon^{-1}x')$ and define a partial mollification of $v$ with respect to $z'$ as
\begin{align*}
\hat v^\epsilon(t,x',x^d)&=\int_{\bR}\int_{\bR^{d-1}}v(s,y',x^d) \zeta_\epsilon(t-s,x'-y')\,dy'\,ds\\
&=\int_{\bR} \int_{\bR^{d-1}}v(t-\epsilon^2s,x'-\epsilon y',x^d) \zeta(s,y')\,dy'\,ds.
\end{align*}

By virtue of Taylor's formula, it is not hard to prove the following lemma for partial mollifications (see, e.g., \cite[Chapter 3]{Kr96}).
\begin{lemma}
                                    \label{lem2.25}
Let $T\in (-\infty,\infty]$ and $\delta\in [0,2]$.
Suppose $v\in C^{\delta/2,\delta}_{z'}(\bR^{d+1}_T)$. Then for any $\epsilon>0$,
\[
\epsilon^{2-\delta}\sup_{\bR^{d+1}_T} |D_t \hat v^\epsilon|+\epsilon^{1-\delta} \sup_{\bR^{d+1}_T} |D_{x'} \hat v^\epsilon| +\epsilon^{2-\delta} \sup_{\bR^{d+1}_T} |D_{x'}^2 \hat v^\epsilon| \le N[v]_{z',\delta/2,\delta},
\]
\[
\sup_{\bR^{d+1}_T}|v-\hat v^\epsilon| \le N(d,\delta,\eta)\epsilon^{\delta}[v]_{z',\delta/2,\delta},
\]
where $N=N(d,\delta,\eta)$.
\end{lemma}

Now we are ready to estimate $\ip{u}_{1+\delta}$.
First assume that \eqref{parabolic} holds in $\bR^{d+1}_0:=\{(t,x)\in \bR^{d+1}:t<0\}$ and $u\in C^{(1+\delta)/2,1+\delta}_{z'}(\bR^{d+1}_0)$. Define
$$
\|u\|_{\text{unif},L_{2,\text{loc}}}:=\sup_{Q_1(z_1)\subset \bR^{d+1}_0}\|u\|_{L_2(Q_1(z_1))}.
$$
We claim
\begin{equation}
                                            \label{eq2.14}
[u]_{z',(1+\delta)/2,1+\delta}\le N_2(|g|_{z',\delta/2,\delta}+|f|_{0}+\|u\|_{\text{unif},L_{2,\text{loc}}}),
\end{equation}
where $N_2=N_2(d,n,\nu,\delta,K,[A]_{z',\delta/2,\delta},[B]_{z',\delta/2,\delta})>0$.
Take $r>0$. Let $\kappa>4$ be a number to be chosen later. Denote
$$
\cP_0 u=-u_t+D_\alpha(A^{\alpha\beta}(0,x^d)D_\beta u).
$$
Then we have
$$
\cP_0 u=\Div (\fg+m)+\ff,
$$
where
$$
m_\alpha^i(z)=(A^{\alpha\beta}_{ij}(0,x^d)-A^{\alpha\beta}_{ij}(z))
D_\beta u^j,
$$
$$
\fg=(\fg_\alpha),\quad\fg_\alpha:=g_\alpha-B^\alpha u,\quad
\ff:=f-\hat B^\alpha D_\alpha u-Cu.
$$
Note that
\begin{equation}
                                            \label{eq5.53}
|m|\le N|z'|^{\delta}[A]_{z',\delta/2,\delta}|Du|_0.
\end{equation}
Moreover, since the coefficients of $\cP_0$ are independent of $z'$, we have
\[
\cP_0 \hat u^{\kappa r}=\Div (\hat{\fg}{}^{\kappa r}+\hat{m}{}^{\kappa r})+\hat{\ff}{}^{\kappa r}.
\]

Let $v\in \cH^1_2(Q_{\kappa r})$ be a weak solution to the equation
\[
\left\{
  \begin{aligned}
    \cP_0 v= \Div(\fg-\hat{\fg}{}^{\kappa r}+m-\hat{m}{}^{\kappa r})+\ff-\hat{\ff}{}^{\kappa r} \quad & \hbox{in $Q_{\kappa r}$;} \\
    v=0 \quad & \hbox{on $\partial_p Q_{\kappa r}$.}
  \end{aligned}
\right.
\]
By Lemma \ref{lem3.6}, we get
\begin{equation}
                                \label{eq3.10vp}
\|v\|_{L_2(Q_{\kappa r})}\le N\kappa r\|\fg-\hat{\fg}{}^{\kappa r}+m-\hat{m}{}^{\kappa r}\|_{L_2(Q_{\kappa r})}
+N(\kappa r)^{2}\|\ff-\hat{\ff}{}^{\kappa r}\|_{L_2(Q_{\kappa r})}.
\end{equation}
Corollary \ref{cor3.1} and a scaling argument yield
\[
|v|_{0;Q_{\kappa r/2}}\le N\kappa r |\fg-\hat{\fg}{}^{\kappa r}+m-\hat{m}{}^{\kappa r}|_{0;Q_{\kappa r}}
\]
\[
+ N(\kappa r)^2 |\ff-\hat{\ff}{}^{\kappa r}|_{0;Q_{\kappa r}}
+N(\kappa r)^{-(d+2)/2}\|v\|_{L_2(Q_{\kappa r})}.
\]
This together with \eqref{eq3.10vp}, \eqref{eq5.53} and Lemma~\ref{lem2.25} yields
\begin{equation}
                                                \label{eq3.14p}
|v|_{0;Q_{\kappa r/2}}\le N(\kappa r)^{1+\delta} ([\fg]_{z',\delta/2,\delta}+[A]_{z',\delta/2,\delta}|Du|_0+(\kappa r)^{1-\delta}|\ff|_{0}).
\end{equation}

Let $w=u-\hat{u}{}^{\kappa r}-v$. Then $w$ satisfies $\cP_0 w = 0$ in $Q_{\kappa r}$. It follows from Lemma \ref{lem3.2}, Lemma \ref{lem3.1} and a scaling that
\begin{equation*}
|w-\hT^1_{z_0'}w|_{0;Q_r} \le N r^2[w]_{z',1,2;Q_r}\le N\kappa^{-2} |w|_{0;Q_{\kappa r/2}}
\end{equation*}
With the triangle inequality, \eqref{eq3.14p} and Lemma \ref{lem2.25}, we further deduce
\begin{align}
                                            \label{eq19.09}
|w-\hT^1_{z_0'}w|_{0;Q_r} &\le N\kappa^{-2} (|v|_{0;Q_{\kappa r/2}}+|u-\hat{u}{}^{\kappa r}|_{0;Q_{\kappa r/2}})\nonumber\\
&\le N\kappa^{\delta-1}r^{1+\delta}
([\fg]_{z',\delta/2,\delta}+[A]_{z',\delta/2,\delta}|Du|_0
+(\kappa r)^{1-\delta}|\ff|_{0})\nonumber\\
&\,\,+N\kappa^{\delta-1}r^{1+\delta}[u]_{z',(1+\delta)/2,1+\delta}.
\end{align}

By Lemma~\ref{lem2.25}, we also get
\begin{align}
                                                \label{eq3.17p}
|\hat u^{\kappa r}-\hT^1_{z_0'} \hat u^{\kappa r}|_{0;Q_r}
&\le Nr^2 ([D_{x'}^2 \hat u^{\kappa r}]_{0;Q_r}+[D_t \hat u^{\kappa r}]_{0;Q_r})\\
\nonumber
&\le N\kappa^{\delta-1} r^{1+\delta}[u]_{z',(1+\delta)/2,1+\delta}.
\end{align}
Take $p=\hT^1_{z_0'}w+\hT^1_{z_0'} \hat u^{\kappa r}\in\hat\bP_1$. Then
from \eqref{eq3.14p}, \eqref{eq19.09} and \eqref{eq3.17p}, we get
\begin{align*}
|u-p|_{0;Q_r}
&\le |v|_{0;Q_r}+|\hat u^{\kappa r}-\hT^1_{z_0'} \hat u^{\kappa r}|_{0;Q_r}+|w-\hT^1_{z_0'}w|_{0;Q_r}\\
&\le N(\kappa r)^{1+\delta} ([\fg]_{z',\delta/2,\delta}+[A]_{z',\delta/2,\delta}|Du|_0+(\kappa r)^{1-\delta}|\ff|_{0})\\
&\,\,+N\kappa^{\delta-1} r^{1+\delta}[u]_{z',(1+\delta)/2,1+\delta}.
\end{align*}
By a shift of the coordinates, we have for any $z_0\in \bR^{d+1}_0$ and $r>0$,
\begin{multline}
                                \label{eq2.28p2}
r^{-1-\delta}\,\inf_{p\in\hat\bP_1}|u-p|_{0;Q_r(z_0)}\\
\le N\kappa^{1+\delta} ([\fg]_{z',\delta/2,\delta}+[A]_{z',\delta/2,\delta}|Du|_0+(\kappa r)^{1-\delta}|\ff|_{0})+
N_1\kappa^{\delta-1} [u]_{z',(1+\delta)/2,1+\delta},
\end{multline}
where $N_1=N_1(d,n,\nu)>0$.
We fix a large $\kappa=\kappa(d,n,\nu)$ such that $N_1\kappa^{\delta-1}<1/2$.
For any $r<1/\kappa$, we get from \eqref{eq2.28p2}
\begin{align*}
&r^{-1-\delta}\,\inf_{p\in\hat\bP_1}|u-p|_{0;Q_r(z_0)}\\
&\le N([\fg]_{z',\delta/2,\delta}+[A]_{z',\delta/2,\delta}|Du|_0+|\ff|_{0})+
1/2[u]_{z',(1+\delta)/2,1+\delta}.
\end{align*}
By the definition of $\ff$ and $\fg$,
\begin{align}
                                \label{eq21.28}
&r^{-1-\delta}\,\inf_{p\in\hat\bP_1}|u-p|_{0;Q_r(z_0)}\nonumber\\
&\le N([g]_{z',\delta/2,\delta}+|u|_0+[u]_{z',\delta/2,\delta}
+|Du|_0+|f|_{0})+1/2[u]_{z',(1+\delta)/2,1+\delta}\nonumber\\
&\le N_2(|g|_{z',\delta/2,\delta}+|f|_{0}+\|u\|_{\text{unif},L_{2,\text{loc}}})
+1/2[u]_{z',(1+\delta)/2,1+\delta},
\end{align}
where $N_2=N_2(d,n,\nu,\delta,K,[A]_{z',\delta/2,\delta},[B]_{z',\delta/2,\delta})>0$.
In the last inequality, we used Corollary \ref{cor3.1} and Theorem \ref{thm1}.
On the other hand, for $r\ge 1/\kappa$, clearly we have
\begin{equation*}
r^{-1-\delta}\,\inf_{p\in\hat\bP_1}|u-p|_{0;Q_r(z_0)}\le N|u|_0.
\end{equation*}
Thus, \eqref{eq21.28} holds in any case.

Now by first taking the supremum in \eqref{eq21.28} with respect to $r$ and $z_0$, and then using the equivalence of parabolic H\"older semi-norms similar to \cite[Theorem 8.5.2]{Kr96}, we obtain
\[
[u]_{z',(1+\delta)/2,1+\delta}\le N_2(|g|_{z',\delta/2,\delta}+|f|_{0}+\|u\|_{\text{unif},L_{2,\text{loc}}})
+1/2[u]_{z',(1+\delta)/2,1+\delta},
\]
which implies \eqref{eq2.14} under the assumption that $u\in C_{z'}^{(1+\delta)/2,1+\delta}(\bR^{d+1}_0)$.

Next, we localize the estimate. Let $u$ be a weak solution to \eqref{parabolic} in $Q_1$ and assume $u\in C_{z'}^{(1+\delta)/2,1+\delta}(Q_{3/4})$. We take a cutoff function $\varsigma$ such that $\varsigma=0$ in $\bR^{d+1}_0\setminus Q_{3/4}$ and $\varsigma=1$ in $Q_{1/2}$. Then $\tilde u=u\varsigma$ satisfies
$$
\cP \tilde u=\Div \tilde g+\tilde f\quad\text{in }\bR^{d+1}_0,
$$
where
\begin{align*}
\tilde g_\alpha&=g_\alpha\varsigma+A^{\alpha\beta}uD_\beta\varsigma,\\
\tilde f&=f\varsigma-u\varsigma_t-g_\alpha D_\alpha \varsigma
+\left(A^{\alpha\beta}D_\beta u+(B^\alpha+\hat B^\alpha) u\right)D_\alpha\varsigma .
\end{align*}
We then have \eqref{eq2.14} with $\tilde u,\tilde f,\tilde g$ in place of $u,f,g$. This together with Theorem \ref{thm1} and Corollary \ref{cor3.1} yields
\begin{equation}
[u]_{z',(1+\delta)/2,1+\delta;Q_{1/2}}\le N_2\left(|g|_{z',\delta/2,\delta;Q_1}+|f|_{0;Q_1}
+\|u\|_{L_2(Q_1)}\right).
                                    \label{eq2.37}
\end{equation}
Finally, we drop the assumption that $u\in C_{z'}^{(1+\delta)/2,1+\delta}(Q_{3/4})$ by using the same approximation argument as in the proof of Theorem \ref{thm1}. Combining \eqref{eq11.55} and \eqref{eq2.37}, the theorem is proved.

\mysection{Non-divergence form equations}
                    \label{sec6}

In this section, we are concerned with non-divergence scalar parabolic equations
\begin{equation}
                                                \label{eq21.28b}
P u:=-u_t+a^{\alpha\beta}D_{\alpha\beta}u+b^\alpha D_\alpha u+c u=f,
\end{equation}
The coefficients $a^{\alpha\beta}$, $b^\alpha$, and $c$ are assumed to be measurable and bounded by $K$, and
the leading coefficients $a^{\alpha\beta}$ are uniformly elliptic with ellipticity constant $\nu$:
$$
\nu|\xi|^2\le a^{\alpha\beta}\xi^\alpha\xi^\beta,\quad |a^{\alpha\beta}|\le \nu^{-1}
$$
for any $\xi\in \bR^d$.
The aim here is to prove estimates in the same spirit as in Theorems \ref{thm1} and \ref{thm2} for non-divergence form equations. We state the main results of this section as follows. Theorem \ref{thm4} improves Theorem 2.14 \cite{DongSeickK09} and Theorem 3.1 \cite{TianWang} in the case $q=d-1$.

\begin{theorem}
                            \label{thm3}
Let $a,b,c\in C_{z'}^{\text{Dini}}$ and $f\in C_{z'}^{\text{Dini}}(Q_1)$. Assume that $u\in W^{1,2}_2(Q_1)$ is a strong solution to \eqref{eq21.28b} in $Q_1$. Then we have
\begin{equation}
                                    \label{eq21.58b}
|u|_{1,2;Q_{1/2}}\le N(I[\omega_{f,z'}](1)+|f|_{0;Q_1}+\|u\|_{L_2(Q_1)}),
\end{equation}
where $N=N(d,\nu,K,\omega_{a,z'},\omega_{b,z'},\omega_{c,z'})$. Moreover, $u_t$ and $D_{xx'}u$ are continuous in $\overline{Q_{1/2}}$.
\end{theorem}

\begin{theorem}
                            \label{thm4}
Let $\delta\in (0,1)$, $a,b,c\in C_{z'}^{\delta/2,\delta}$ and $f\in C_{z'}^{\delta/2,\delta}(Q_1)$. Assume that $u\in W^{1,2}_2(Q_1)$ is a strong solution to \eqref{eq21.28b} in $Q_1$. Then we have
\begin{equation*}
|u|_{1,2;Q_{1/2}}+[u_t]_{\delta/2,\delta;Q_{1/2}}+
[D_{xx'}u]_{\delta/2,\delta;Q_{1/2}}\le N(|f|_{z',\delta/2,\delta;Q_{1}}+\|u\|_{L_2(Q_1)}),
\end{equation*}
where $N=N(d,\delta,\nu,K,[a]_{z',\delta/2,\delta},[b]_{z',\delta/2,\delta},
[c]_{z',\delta/2,\delta})$.
\end{theorem}

\begin{remark}
From Theorems \ref{thm3} and \ref{thm4}, we also obtain the corresponding results for non-divergence form elliptic equations, as in Corollaries \ref{cor1e} and \ref{cor2e}.
\end{remark}

\begin{remark}
We consider the case when the coefficients and data are piecewise H\"older continuous. Suppose that $Q_1$ is divided into $M$ laminate sub-domains $\cD_1,\cD_2,\cdots,\cD_M$ by $M-1$ parallel hyperplanes with the common normal direction $(0,\cdots,0,1)$. Assume $a,b,c$ and $f$ are in $C^{\delta/2,\delta}(\cD_i)$ for each $i=1,2,\cdots,M$, but may have jump discontinuities across these hyperplanes. By Theorem \ref{thm4} and a covering argument, for any $\epsilon\in (0,1/2)$, $u_t$ and $D_{xx'}u$ are in $C^{\delta/2,\delta}(Q_{1-\epsilon})$. Moreover, by the same theorem, $u$ and $Du$ are in $C^{\delta/2,\delta}(Q_{1-\epsilon})$. Now restricted to each $\cD_i\cap Q_{1-\epsilon}$, since
$$
D_{dd}u=(a^{dd})^{-1}\left(f+u_t-b^\alpha D_\alpha u-cu-\sum_{(\alpha,\beta)\neq (d,d)}a^{\alpha\beta}D_{\alpha\beta}u\right),
$$
we conclude that $D_{dd}u\in C^{\delta/2,\delta}(\cD_i\cap Q_{1-\epsilon})$, and hence $u\in C^{1+\delta/2,2+\delta}(\cD_i\cap Q_{1-\epsilon})$.
It is worth noting that the $C^{1+\delta/2,2+\delta}$ norm of in each sub-domain $\cD_i\cap Q_{1-\epsilon}$ is independent the number $M$.
\end{remark}

The proofs of Theorems \ref{thm3} and \ref{thm4} follow the line of the proofs of Theorems \ref{thm1} and \ref{thm2}, and in fact are simpler. First, by using the argument in Section \ref{sec3}, we obtain from the main result of \cite{KK2} the following interior estimates, which is analogous to Corollary \ref{cor3.1}.

\begin{lemma}
                                    \label{lem6.3}
Let $q\in [2,\infty)$. Assume $a^{\alpha\beta}$ are VMO in $z'$, $u\in C^\infty_{\text{loc}}$ satisfies \eqref{eq21.28b}
in $Q_1$, where $f\in L_q(Q_1)$. Then there exists a constant $N=N(d,\nu,K,\omega_{a,z'},q)$ such that
\begin{equation*}
\|u\|_{W^{1,2}_q(Q_{1/2})}\le N(\|u\|_{L_2(Q_1)}+\|f\|_{L_q(Q_1)}).
\end{equation*}
In particular, if $q>d+2$, it holds that
$$
|u|_{(1+\gamma)/2,1+\gamma;Q_{1/2}}\le N(\|u\|_{L_2(Q_1)}+\|f\|_{L_q(Q_1)}),
$$
where $\gamma=1-(d+2)/q$.
\end{lemma}

We denote
\begin{equation}
                                    \label{eq16.1.27}
P_0(u)=-u_t+a^{\alpha\beta}D_{\alpha \beta} u.
\end{equation}
The next lemma is an analogy of Lemma \ref{lem3.2}.
\begin{lemma}
                                    \label{lem6.4}
Let $\gamma\in (0,1)$. Assume $u\in C^\infty_{\text{loc}}$ satisfies $P_0 u=f_0$ in $Q_1$, where $f_0$ and $a^{\alpha\beta}$ are independent of $t$ and $x'$. Then there exists a constant $N=N(d,\gamma,\nu)$ such that
\begin{align}
                                    \label{eq22.35}
[u_t]_{(1+\gamma)/2,1+\gamma;Q_{1/2}}&\le N\|u_t\|_{L_2(Q_1)},\\
                                    \label{eq22.36}
[D_{x'} u]_{(1+\gamma)/2,1+\gamma;Q_{1/2}}&\le N\|D_{xx'} u\|_{L_2(Q_1)}.
\end{align}
\end{lemma}
\begin{proof}
Note that in $Q_1$
\begin{equation*}
P_0(D_t^i D_{x'}^j u)=0,
\end{equation*}
wherever $i+j\ge 1$.
Then \eqref{eq22.35} follows from Lemma \ref{lem6.3} applied to $u_t$. For any $1/2\le r<R\le 1$, applying the same lemma to $D_{x'}u-(D_{x'}u)_{Q_R}$ gives
\begin{equation}
                            \label{eq20.53}
[D_{x'} u]_{(1+\gamma)/2,1+\gamma;Q_{r}}\le N\|D_{x'} u-(D_{x'} u)_{Q_R}\|_{L_2(Q_R)}.
\end{equation}
To bound the right-hand side of \eqref{eq20.53},
we use an idea in \cite{Dong09a} to utilize the divergence structure of the equation after making a change of variables. Let
$$
y^d=\varphi(x^d):=\int_0^{x^d} \frac 1 {a^{dd}(s)}\,ds,\quad y^\beta=x^\beta,\,\,\beta=1,\cdots,d-1.
$$
It is easy to see that $\varphi$ is a bi-Lipschitz function and
\begin{equation*}
\delta \le y^d/x^d\le \delta^{-1},\quad D_{y^d}=a^{dd}(x^d)D_{x^d}.
\end{equation*}
Denote
$$
v(t,y',y^d)=u(t,y',\varphi^{-1}(y^d)),\quad
\hat a^{ij}(y^d)=a^{ij}(\varphi^{-1}(y^d)),
$$
$$
\hat f(y^d)=f(\varphi^{-1}(y^d)).
$$
Define a divergence form operator $\hat P_0$ by
$$
\hat P_0 v=-v_t+D_{d}\left(\frac 1 {\hat a^{dd}} D_d v\right)+\sum_{\beta=1}^{d-1} D_{\beta}\left(\frac {\hat a^{d\beta}+\hat a^{\beta d}}
{\hat a^{dd}} D_dv
\right)+\sum_{\alpha,\beta=1}^{d-1} D_{\alpha}(\hat a^{\alpha\beta}D_\beta v).
$$
Clearly, $\tilde P_0$ is uniformly nondegenerate and $v$ satisfies $\hat P_0 v=\hat f$ in some stretched cylindrical domain. Since $\hat P_0 (D_{y'}v)=0$, one can use Lemma \ref{paraPoin} i) to estimate the mean oscillation of $D_{y'}v$ in each parabolic cylinder by the integral of $D_{yy'}v$ in the same cylinder. To finish the proof of \eqref{eq22.36}, it suffices to use a covering argument bearing in mind that the metrics in $x$-coordinate and $y$-coordinate are comparable.
\end{proof}

Now we are ready to prove Theorems \ref{thm3} and \ref{thm4}.

\begin{proof}[Proof of theorem \ref{thm3}]
First we assume that $b=c=0$. Also, arguing as before we may assume that $a$ are infinitely differentiable and $u$ has bounded derivative up to fourth order in $Q_{3/4}$. We take $0<\gamma<1$.
Fix a point $z_0\in Q_{3/4}$, and take $0<r<R\le (3/4-|x_0|)/4$. Now take $z_1'\in Q_{R}'(z_0')$ and denote
$$
P_{z_1'} u=-u_t+A^{\alpha\beta}(z_1',x^d)D_{\alpha\beta}u.
$$
Then we have
$$
P_{z_1'} u=f+m,
$$
where
$$
m(z)=(a^{\alpha\beta}(z_1',x^d)-a^{\alpha\beta}(z))
D_{\alpha\beta} u.
$$

 Let $v$ be the strong solution to the equation
\[
\left\{
  \begin{aligned}
    P_{z_1'} v= f(z)-f(z_1',x^d)+m(z) \quad & \hbox{in $Q_{2R}(z_0)$;} \\
    v=0 \quad & \hbox{on $\partial_p Q_{2R}(z_0)$.}
  \end{aligned}
\right.
\]
We rewrite $P_{z_1'}$ as a divergence form operator as in Lemma \ref{lem6.4}, then use Lemma 3.2 to get
$$
\|v\|_{L_2(Q_{2R}(z_0))}\le NR^2\|f-f(z_1',x^d)+m\|_{L_2(Q_{2R}(z_0))}.
$$
This together with Lemma \ref{lem6.3} gives
\begin{equation}
\|D^2v\|_{L_2(Q_{R}(z_0))}+\|v_t\|_{L_2(Q_{R}(z_0))}
                                \label{eq22.57}
\le N\|f-f(z_1',x^d)+m\|_{L_2(Q_{2R}(z_0))}.
\end{equation}
Let $w=u-v$. Then $w$ satisfies $P_{z_1'} w = f(z_1',x^d)$ in $Q_{R}(z_0)$.
It follows from Lemma \ref{lem6.4} and a suitable scaling that
\begin{align}
&\int_{Q_r(z_0)}|D_{xx'}w-(D_{xx'}w)_{Q_r(z_0)}|^2+|w_t-(w_t)_{Q_r(z_0)}|^2\,dz\nonumber\\
&\,\le N(r/R)^{d+2+2\gamma}
\int_{Q_R(z_0)}|D_{xx'}w|^2+|w_t|^2\,dz.
                                    \label{eq23.16}
\end{align}
Define
$$
h(x^d)=\int_{-1}^{x^d}\int_{-1}^s(a^{dd}(z_1',\tau))^{-1}\left(a^{\alpha\beta}(z_1',\tau)
(D_{\alpha\beta}w)_{Q_R(z_0)}-(w_t)_{Q_R(z_0)}\right)\,d\tau\,ds,
$$
and
$$
\tilde w:=w-t(w_t)_{Q_R(z_0)}-\frac 12x^\alpha x^\beta(D_{\alpha\beta}w)_{Q_R(z_0)}+h(x^d).
$$
Then, we have
$$
D_{xx'} \tilde w=D_{xx'}w-(D_{xx'}w)_{Q_R(z_0)},\quad \tilde w_t=w_t-(w_t)_{Q_R(z_0)}.
$$
It is easily seen that $\tilde w$ also satisfies $P_{z_1'} \tilde w = 0$ in $Q_{R}(z_0)$. We substitute $w$ in \eqref{eq23.16} by $\tilde w$ to get
\begin{align}
&\int_{Q_r(z_0)}|D_{xx'}w-(D_{xx'}w)_{Q_r(z_0)}|^2+|w_t-(w_t)_{Q_r(z_0)}|^2\,dz\nonumber\\
&\,\le N(r/R)^{d+2+2\gamma}
\int_{Q_R(z_0)}|D_{xx'}w-(D_{xx'}w)_{Q_R(z_0)}|^2+|w_t-(w_t)_{Q_R(z_0)}|^2\,dz.
                                    \label{eq23.17}
\end{align}
We combine \eqref{eq22.57} with \eqref{eq23.17} and use the triangle inequality to obtain
\begin{align}
&\int_{Q_r(z_0)}|D_{xx'}u-(D_{xx'}u)_{Q_r(z_0)}|^2+|D_tu-(D_tu)_{Q_r(z_0)}|^2\,dz\nonumber\\
&\,\le N_1(r/R)^{d+2+2\gamma}
\int_{Q_R(z_0)}|D_{xx'}u-(D_{xx'}u)_{Q_R(z_0)}|^2+|D_tu-(D_tu)_{Q_r(z_0)}|^2\,dz\nonumber\\
&\,\,\,+N\|f-f(z_1',x^d)+m\|^2_{L_2(Q_{2R}(z_0))},
                                    \label{eq23.48}
\end{align}
where $N_1=N_1(d,\nu)$. Now we take average of both sides of \eqref{eq23.48} with respect to $z_1'\in Q_R'(z_0')$ to get
\begin{align}
&\int_{Q_r(z_0)}|D_{xx'}u-(D_{xx'}u)_{Q_r(z_0)}|^2+|D_tu-(D_tu)_{Q_r(z_0)}|^2\,dz\nonumber\\
&\,\le N_1(r/R)^{d+2+2\gamma}
\int_{Q_R(z_0)}|D_{xx'}u-(D_{xx'}u)_{Q_R(z_0)}|^2+|D_tu-(D_tu)_{Q_r(z_0)}|^2\,dz\nonumber\\
&\,\,\,+NR^{d+2}\omega_{f,z'}^2(2R)+NR^{d+2}\omega_{a,z'}^2(2R)|D^2u|^2_{0;Q_{2R}(z_0)},
                                    \label{eq23.48bb}
\end{align}
As in Section \ref{sec4}, we immediately get \eqref{eq21.58b} from \eqref{eq23.48bb} by using an iteration argument and Lemma \ref{lem6.3}. Then the argument in Step 2 of the proof of Theorem \ref{thm1} shows the continuity of $D_{xx'}u$ and $u_t$.

In the general case, we move all the lower order terms to the right-hand side:
$$
-u_t+a^{\alpha\beta}D_{\alpha\beta}u=f-b^\alpha D_\alpha u-cu:=\ff.
$$
From the proof above, we get
$$
|u|_{1,2;Q_{1/2}}\le N(I[\omega_{\ff,z'}](1)+|\ff|_{0;Q_1}+\|u\|_{L_2(Q_1)}).
$$
To bound the first two terms on the right-hand side, it suffices to use Lemma \ref{lem6.3} and the assumptions on $b$ and $c$.
\end{proof}

\begin{proof}[Proof of Theorem \ref{thm4}]
First we assume that $b=c=0$. We fix a point $z_0\in Q_{1/4}$ and take $0<r<R\le 1/8$. Following the proof of Theorem \ref{thm3}, we take $\gamma\in (\delta,1)$. Owing to \eqref{eq23.48}, we get
\begin{align}
&\int_{Q_r(z_0)}|D_{xx'}u-(D_{xx'}u)_{Q_r(z_0)}|^2+|u_t-(u_t)_{Q_r(z_0)}|^2\,dz\nonumber\\
&\,\le N_1(r/R)^{d+2+2\gamma}
\int_{Q_R(z_0)}|D_{xx'}u-(D_{xx'}u)_{Q_R(z_0)}|^2+|u_t-(u_t)_{Q_R(z_0)}|^2\,dz\nonumber\\
&\,\,\,+N_2\left([f]_{z',\delta/2,\delta;Q_1}
+|D^2u|_{0;Q_{2R}(z_0)}\right)^2R^{d+2+2\delta},
                                    \label{eq11.16c}
\end{align}
where $N_1$ only depends on $d$ and $\nu$, and $N_2$ also depends on the $C_{z'}^{\delta/2,\delta}$ semi-norm of $a$.
Since \eqref{eq11.16c} holds for any $0<r<R\le 1/8$ and $\delta<\gamma$, by a well-known iteration argument
(see e.g., \cite[Lemma 2.1, p. 86]{Giaq83}),
\begin{align}
&\int_{Q_r(z_0)}|D_{xx'}u-(D_{xx'}u)_{Q_r(z_0)}|^2+|u_t-(u_t)_{Q_r(z_0)}|^2\,dz\nonumber\\
&\,\le N_1r^{d+2+2\delta}
\int_{Q_{1/4}(z_0)}|D_{xx'}u-(D_{xx'}u)_{Q_{1/4}(z_0)}|^2+|u_t-(u_t)_{Q_{1/4}(z_0)}|^2\,dz\nonumber\\
&\,\,\,+N_2r^{d+2+2\delta}\left([f]_{z',\delta/2,\delta;Q_1}
+|D^2u|_{0;Q_{2R}(z_0)}\right)^2.
                                    \label{eq11.39c}
\end{align}
We get from \eqref{eq11.39c}, Theorem \ref{thm3} and Lemma  \ref{lem6.3} that
\begin{align}
&\int_{Q_r(z_0)}|D_{xx'}u-(D_{xx'}u)_{Q_r(z_0)}|^2+|u_t-(u_t)_{Q_r(z_0)}|^2\,dz\nonumber\\
&\,\le N_2r^{d+2+2\delta}\left(|f|_{z',\delta/2,\delta;Q_1}+\|u\|_{L_2(Q_1)}\right)^2.
                                    \label{eq11.45c}
\end{align}
Since \eqref{eq11.45c} holds for any $r\in (0,1/8)$ and $z_0\in Q_{1/4}$, by Campanato's characterization of H\"older continuous functions, we obtain
\begin{equation*}
[D_{xx'}u]_{\delta/2,\delta;Q_{1/4}}+[u_t]_{\delta/2,\delta;Q_{1/4}}\le N_2\left(|f|_{z',\delta/2,\delta;Q_1}+\|u\|_{L_2(Q_1)}\right).
\end{equation*}
This together with a dilation and covering argument gives
\begin{equation*}
[D_{xx'}u]_{\delta/2,\delta;Q_{1/2}}+[u_t]_{\delta/2,\delta;Q_{1/2}}\le N_2\left(|f|_{z',\delta/2,\delta;Q_1}+\|u\|_{L_2(Q_1)}\right).
\end{equation*}

In the general case, we move all the lower order terms to the right-hand side and ague as before.
\end{proof}

We finish this section by proving the following ``partial'' Schauder estimates. These results  generalize Theorem 5.1 \cite{DongSeickK09} and Theorem 2.1 \cite{TianWang} for the Poisson equation.

Let us introduce a few more notation. Let $q$ be an integer such that $1\le q\le d-1$. We distinguish the first $q$ coordinates of $x$ from the rest and write $x=(\tilde x,\check x)$, where $\tilde x=(x^1,\cdots,x^q)$ and $\check x=(x^{q+1},\cdots,x^d)$. We also denote $\tilde z=(t,\tilde x)$. As in Section \ref{sec2}, we introduce the partial Dini continuous space $C_{\tilde z}^{\text{Dini}}(\cD)$ and the partial H\"older space $C_{\tilde z}^{\delta/2,\delta}$, as well as their corresponding norms.

\begin{theorem}
                            \label{thm5}
Let $a=a(x^d)$ be a measurable function of $x^d$ alone. Let $b,c\in C_{\tilde z}^{\text{Dini}}$ and $f\in C_{\tilde z}^{\text{Dini}}(Q_1)$. Assume that $u\in W^{1,2}_2(Q_1)$ is a strong solution to \eqref{eq21.28b} in $Q_1$. Then we have
\begin{equation*}
|u_t|_{0;Q_{1/2}}+|D_{x\tilde x}u|_{0;Q_{1/2}}\le N(I[\omega_{f,z'}](1)+|f|_{0;Q_1}+\|u\|_{L_2(Q_1)}),
\end{equation*}
where $N=N(d,\nu,K,\omega_{b,\tilde z},\omega_{c,\tilde z})$. Moreover, $u_t$ and $D_{x\tilde x}u$ are continuous in $\overline{Q_{1/2}}$.
\end{theorem}

\begin{theorem}
                            \label{thm6}
Let $\delta\in (0,1)$, $a=a(x^d)$ be a measurable function of $x^d$ alone. Let $b,c\in C_{\tilde z}^{\delta/2,\delta}$ and $f\in C_{\tilde z}^{\delta/2,\delta}(Q_1)$. Assume that $u\in W^{1,2}_2(Q_1)$ is a strong solution to \eqref{eq21.28b} in $Q_1$. Then we have
\begin{equation*}
|u_t|_{\delta/2,\delta;Q_{1/2}}+
|D_{x\tilde x}u|_{\delta/2,\delta;Q_{1/2}}\le N(|f|_{\tilde z,\delta/2,\delta;Q_{1}}+\|u\|_{L_2(Q_1)}),
\end{equation*}
where $N=N(d,\delta,\nu,K,[b]_{\tilde z,\delta/2,\delta},
[c]_{\tilde z,\delta/2,\delta})$.
\end{theorem}

For the proofs of Theorems \ref{thm5} and \ref{thm6}, first we note that similar to Lemma \ref{lem6.4} if $f_0$ is independent of $\tilde z$, then \eqref{eq22.35} still holds and we have
\begin{equation*}
[D_{\tilde x} u]_{(1+\gamma)/2,1+\gamma;Q_{1/2}}\le N\|D_{x\tilde x} u\|_{L_2(Q_1)}.
\end{equation*}
Following the proof of Theorem \ref{thm3}, first we assume there is no lower order terms. Let $v$ be the solution of
\[
\left\{
  \begin{aligned}
    P_0 v= f(z)-f(\tilde z_1,\check x) \quad & \hbox{in $Q_{2R}(z_0)$;} \\
    v=0 \quad & \hbox{on $\partial_p Q_{2R}(z_0)$,}
  \end{aligned}
\right.
\]
where $P_0$ is defined in \eqref{eq16.1.27}.
Then as before we have
\begin{equation}
\|D^2v\|_{L_2(Q_{R}(z_0))}+\|v_t\|_{L_2(Q_{R}(z_0))}
                                \label{eq22.57k}
\le N\|f-f(\tilde z_1,\check x)\|_{L_2(Q_{2R}(z_0))}.
\end{equation}
Moreover, $w=u-v$ satisfies
\begin{align}
&\int_{Q_r(z_0)}|D_{x\tilde x}w-(D_{x\tilde x}w)_{Q_r(z_0)}|^2+|w_t-(w_t)_{Q_r(z_0)}|^2\,dz\nonumber\\
&\,\le N(r/R)^{d+2+2\gamma}
\int_{Q_R(z_0)}|D_{x\tilde x}w-(D_{x\tilde x}w)_{Q_R(z_0)}|^2+|w_t-(w_t)_{Q_R(z_0)}|^2\,dz.
                                    \label{eq23.17k}
\end{align}
Combining \eqref{eq22.57k} with \eqref{eq23.17k} and taking average in $\tilde z_1$ gives
\begin{align*}
&\int_{Q_r(z_0)}|D_{x\tilde x}u-(D_{x\tilde x}u)_{Q_r(z_0)}|^2+|D_tu-(D_tu)_{Q_r(z_0)}|^2\,dz\nonumber\\
&\,\le N_1(r/R)^{d+2+2\gamma}
\int_{Q_R(z_0)}|D_{x\tilde x}u-(D_{x\tilde x}u)_{Q_R(z_0)}|^2+|D_tu-(D_tu)_{Q_r(z_0)}|^2\,dz\nonumber\\
&\,\,\,+N\omega_{f,\tilde z}^2(2R)R^{d+2}.
\end{align*}
Now a standard iteration argument finishes the proof.

We remark in Theorems \ref{thm5} and \ref{thm6} 
if $b,c$ and $f$ are assumed to be regular only with respect to $\tilde x$, we can still get the estimate of $D_{x\tilde x}u$ by dropping the $u_t$ terms in \eqref{eq22.57k} and \eqref{eq23.17k} and replacing $f(\tilde z_1,\check x)$ by $f(t,\tilde x_1,\check x)$.

\mysection{Appendix}

In the appendix, we give a proof of Lemma \ref{lem3.1}. Let $\lambda_0$ be the constant in Theorem 2.2 of \cite{DK09}. Let
$$
r_k = 1-2^{-k},\quad Q_k = (-r^2_k, 0) \times B_{r_k},
\quad
k = 1, 2, \cdots.
$$
Then we find $\zeta_k(t,x) \in C_0^{\infty}(\bR^{d+1})$ such that
$$
\zeta_k
= \left\{\begin{aligned}
1 \quad &\text{on} \quad Q_k\\
0 \quad &\text{on} \quad \bR^{d+1} \setminus (-r_{k+1}^2, r_{k+1}^2) \times B_{r_{k+1}}
\end{aligned}\right.
$$
and
$$
| D\zeta_k | \le N 2^{k},
\quad
| (\zeta_k)_t | \le N 2^{2k},
\quad
| D^2 \zeta_k | \le N 2^{2k}.
$$
Observe that, for $\lambda_k\ge \lambda_0$,
$$
(\cP_0 - \lambda_k) (\zeta_k u)
= \Div g_k + f_k
\quad \text{in}
\quad
\bR^{d+1}_0,
$$
where
$$
g_k = \left({g_k}_{\alpha} \right),
\quad
{g_k}_{\alpha}= \zeta_k g_{\alpha} + \sum_{\beta=1}^d A^{\alpha\beta} uD_{\beta}\zeta_k,
$$
$$
f_k = \zeta_k f - \sum_{\alpha=1}^d g_{\alpha}D_{\alpha} \zeta_k + \sum_{\alpha,\beta=1}^d A^{\alpha\beta}D_{\beta}uD_{\alpha} \zeta_k
-uD_t \zeta_k - \lambda_k u\zeta_k .
$$
Then by Theorem 2.2 of \cite{DK09},
\begin{align*}
&\| D \left(\zeta_k u \right) \|_{L_p(\bR^{d+1}_0)}
\le N \left( \| g_k \|_{L_p(\bR^{d+1}_0)}
+ \lambda_k^{-1/2} \| f_k \|_{L_p(\bR^{d+1}_0)} \right)\\
&\le N \left(2^k + \lambda_k^{-1/2} 2^{2k} + \lambda_k^{1/2} \right) \|u\|_{L_p(Q_1)}
+ N \lambda_k^{-1/2}  \| f \|_{L_p(Q_1)}\\
&\quad+ N \left( 1 + \lambda_k^{-1/2} 2^k \right) \| g \|_{L_p(Q_1)}
+ N \lambda_k^{-1/2} 2^k\| D \left( \zeta_{k+1} u \right) \|_{L_p(\bR^{d+1}_0)}.
\end{align*}
Set
$$
\cA_k = \| D \left(\zeta_k u\right)\|_{L_p(\bR^{d+1}_0)},
\,\,
\cB = \| u \|_{L_p(Q_1)},
\,\,
\cG = \| g \|_{L_p(Q_1)},
\,\,
\cF = \| f \|_{L_p(Q_1)}.
$$
Then
\begin{align*}
\cA_k
&\le N \left( 2^{k} + \lambda_k^{-1/2} 2^{2k} + \lambda_k^{1/2} \right) \cB\\
&\quad+ N \left(1 + \lambda_k^{-1/2} 2^{k} \right) \cG
+ N \lambda_k^{-1/2}\cF
+ N \lambda_k^{-1/2}2^{k}\cA_{k+1}.
\end{align*}
By multiplying $\varepsilon^k$ both sides and summing up with respect to $k$, we have
\begin{multline*}
\sum_{k=1}^{\infty} \varepsilon^k \cA_k
\le N \cB \sum_{k=1}^{\infty}\left(2^{k} + \lambda_k^{-1/2}2^{2k} + \lambda_k^{1/2} \right) \varepsilon^k\\
+ N \cG \sum_{k=1}^{\infty} \left( 1 + \lambda_k^{-1/2}2^{k} \right) \varepsilon^k
+ N \cF \sum_{k=1}^{\infty} \lambda_k^{-1/2}\varepsilon^k
+ N_1 \sum_{k=1}^{\infty} \lambda_k^{-1/2}(2\varepsilon)^k  \cA_{k+1}.
\end{multline*}
We may certainly assume $N_1\ge \lambda_0^{1/2}$.
Now set
$$
\varepsilon = 1/4,\quad \lambda_k^{1/2} = N_12^{k+2}.
$$
Then
\begin{align*}
&\sum_{k=1}^{\infty}\left( 2^{k} + \lambda_k^{-1/2}2^{2k} + \lambda_k^{1/2} \right) \varepsilon^k\le N,\quad
\sum_{k=1}^{\infty} \left( 1 + \lambda_k^{-1/2}2^{k} \right) \varepsilon^k
\le N,\\
&\sum_{k=1}^{\infty} \lambda_k^{-1/2}\varepsilon^k
\le N,\quad
N_1 \sum_{k=1}^{\infty} \lambda_k^{-1/2}(2\varepsilon)^k  \cA_{k+1}
= \sum_{k=1}^{\infty} \varepsilon^{k+1} \cA_{k+1}
= \sum_{k=2}^{\infty} \varepsilon^{k} \cA_{k}.
\end{align*}
Therefore,
\begin{equation}							\label{eq_001}
\sum_{k=1}^{\infty} \varepsilon^k \cA_k
\le N(\cB + \cG +\cF) + \sum_{k=2}^{\infty} \varepsilon^k \cA_k.
\end{equation}
On the other hand the summations above are finite because
$$
\cA_k \le N 2^{k}\|u\|_{L_p(Q_1)} + N \|D u\|_{L_p(Q_1)}.
$$
Then the inequality \eqref{eq_001} implies that
$$
\| D u \|_{L_p(Q_{1/2})} \le \cA_1 \le N (\cB+\cG + \cF)
= N(\| u \|_{L_p(Q_1)} +\| g \|_{L_p(Q_1)} +\| f \|_{L_p(Q_1)}).
$$
Finally, the estimate of $\|u_t\|_{\bH^{-1}_p(Q_{1/2})}$ follows from the above estimate and
the equation \eqref{parabolic} itself. The lemma is proved.

\section*{Acknowledgement}
The author is grateful to Nicolai V. Krylov, Yanyan Li, and Mikhail V. Safonov for their interests in this work and helpful comments. He would also like to thank the referee for a very careful reading of the manuscript and many useful comments.


\end{document}